\documentclass[12pt]{amsart}
\usepackage{amssymb,times,epsfig}
\usepackage{enumerate}

\parskip=\medskipamount

\newtheorem*{Thm*}{Theorem}
\newtheorem{Thm}{Theorem}[section]
\newtheorem{Cor}[Thm]{Corollary}
\newtheorem{Prop}[Thm]{Proposition}
\newtheorem{Lemma}[Thm]{Lemma}

\theoremstyle{definition}
\newtheorem{Defn}[Thm]{Definition}
\newtheorem{Notation}[Thm]{Notation}

\newtheorem{Remark}[Thm]{Remark}

\newcommand{\mf}[1]{\mathbb{#1}}
\newcommand{\mc}[1]{\mathcal{#1}}

\DeclareMathOperator{\NC}{\mathrm{NC}}

\newcommand{\abs}[1]{\left\vert#1\right\vert}

\newcommand{\set}[1]{\left\{#1\right\}}

\newtheorem{theorem}{Theorem}[section]

\newtheorem{proposition}[theorem]{Proposition}

\newtheorem{corollary}[theorem]{Corollary}

\newcommand{\utimes}{\kern0.05em\buildrel{\times}\over{\rule{0em}{0.004em}}\kern-0.9em\cup \kern0.2em}
\newcommand{\sutimes}{\mathrel{\kern0em\buildrel{\mathsf{x}}\over{\rule{0em}{0.0em}} \kern-0.35em\cup\kern-0.0em}}

\allowdisplaybreaks[1]

\title[Operator-Valued Monotone Convolution Semigroups]{Operator-Valued Monotone Convolution Semigroups and an Extension of the Bercovici-Pata Bijection.}
\author{Michael Anshelevich}
\thanks{The first author was supported in part by NSF grant DMS-1160849.}
\address{Department of Mathematics, Texas A\&M University, College Station, TX 77843-3368}
\email{manshel@math.tamu.edu}
\author{John D. Williams}
\address{Universit{\"a}t des Saarlandes, Fachrichtung Mathematik.  Postfach 151150, 66041. Saarbr{\"u}cken}
\email{williams@math.uni-sb.de}
\date{\today}

\begin{document}

\begin{abstract}
In a 1999 paper, Bercovici and Pata showed that a natural bijection between the classically, free and Boolean infinitely divisible measures held at the level of limit theorems of triangular arrays.  This result was extended to include monotone convolution by the authors in \cite{Ans-Williams-Chernoff}.  In recent years, operator-valued versions of free, Boolean and monotone probability have also been developed. Belinschi, Popa and Vinnikov showed that the Bercovici-Pata bijection holds for the operator-valued versions of free and Boolean probability.  In this article, we extend the bijection to include monotone probability theory even in the operator-valued case. To prove this result, we develop the general theory of composition semigroups of non-commutative functions and largely recapture Berkson and Porta's classical results on composition semigroups of complex functions in operator-valued setting. As a biproduct, we deduce that operator-valued monotonically infinitely divisible distributions belong to monotone convolution semigroups. Finally, in the appendix, we extend the result of the second author on the classification of Cauchy transforms for non-commutative distributions to the Cauchy transforms associated to more general completely positive maps.
\end{abstract}

\maketitle
\section{Introduction}
It is a remarkable fact that there are natural bijections between the classes of infinitely divisible measures in each of the four universal non-commutative probability theories, which not only arise from the L\`{e}vy-Hin\u{c}in representations of the measures, but are maintained at the level of limit theorems of triangular arrays.  This is made precise in the following theorem:
\begin{theorem}\label{scalartheorem}
Fix a finite positive Borel measure $\sigma$ on $\mathbb{R}$, a real number $\gamma$, a sequence of probability measures $\set{\mu_{n}}_{n \in \mf{N}}$, and a sequence of positive integers
$k_{1} < k_{2} < \cdots $. The following assertions are equivalent:
\begin{enumerate} \item \label{classical0}
(Classical / tensor) The sequence $\underbrace{\mu_{n} \ast \mu_{n} \ast \cdots \ast \mu_{n}}_{k_{n}}$ converges weakly to $\nu_{\ast}^{\gamma,\sigma}$;
 \item \label{free0}
(Free) The sequence $\underbrace{\mu_{n} \boxplus \mu_{n} \boxplus \cdots \boxplus \mu_{n}}_{k_{n}}$ converges weakly to $\nu_{\boxplus}^{\gamma,\sigma}$;
 \item \label{boolean0}
(Boolean) The sequence $\underbrace{\mu_{n} \uplus \mu_{n} \uplus \cdots \uplus \mu_{n}}_{k_{n}}$ converges weakly to $\nu_{\uplus}^{\gamma,\sigma}$;
 \item \label{monotone0}
(Monotone) The sequence $\underbrace{\mu_{n} \rhd \mu_{n} \rhd \cdots \rhd \mu_{n}}_{k_{n}}$ converges weakly to $\nu_{\rhd}^{\gamma,\sigma}$;
\item \label{integral0}
The measures
\[
k_{n}\frac{x^{2}}{x^{2}+1}d\mu_{n}(x) \rightarrow \sigma
\]
weakly, and
\[
\lim_{n \uparrow \infty} k_{n} \int_{\mathbb{R}} \frac{x}{x^{2} + 1}d\mu_{n}(x) = \gamma.
\]
\end{enumerate}
\end{theorem}
Here $\nu_\ast^{\gamma,\sigma}$, $\nu_\boxplus^{\gamma,\sigma}$, $\nu_\uplus^{\gamma,\sigma}$, $\nu_\rhd^{\gamma,\sigma}$ are probability measures defined explicitly through their complex-analytic transforms. The equivalence of \eqref{classical0}, \eqref{free0}, \eqref{boolean0} and \eqref{integral0} was proven in a by now classic paper due to Bercovici and Pata \cite{BerPatDomains}.  The inclusion of part \eqref{monotone0}  was proven in our recent paper \cite{Ans-Williams-Chernoff}.

Voiculescu developed operator-valued notions of non-commutative probability \cite{Voiculescu_Mult} where probability measures are replaced by certain completely positive maps from the ring of non-commutative polynomials over a C$^{\ast}$-algebra.
An analogous theorem in this more general setting, namely the equivalence of
parts \eqref{free0} and \eqref{boolean0}, was proven in \cite{Belinschi-Popa-Vinnikov-ID}.  The first main result in this paper is the inclusion of \eqref{monotone0}
at this level of generality.

In order to study monotone infinitely divisible $\mathcal{B}$-valued distributions, we must first develop the theory of composition semigroups of non-commutative functions in a manner analogous to Berkson and Porta's study of these semigroups at the level of complex functions \cite{BP78}.  This stems from the fact that the convolution operation for monotone probability theory satisfies the following relation for the associated $F$-transforms,
\[
F_{\mu \rhd \nu} = F_{\mu} \circ F_{\nu},
\]
so that  infinitely-divisible distributions form such a composition semigroup. In the second main result of the paper, we prove that any monotone infinitely-divisible distribution can be included in such a semigroup. Note that even in the scalar-valued case, this is a recent result, proved by Serban Belinschi in his thesis. Finally, we characterize generators of such composition semigroups, and a smaller set of generators of composition semigroups of $F$-transforms.

In Section \ref{prelim}, we provide background and  preliminary results.  In section \ref{analysis}, we study composition semigroups of vector-valued and non-commutative analytic functions.  The main results of this section are Proposition \ref{differentiate0}, which shows that there is a natural notion of a time derivative for semigroups of vector-valued analytic functions $\{f_{t}\}_{t\geq 0}$, and Theorem \ref{generation},  which proves that, in the case of $F$-transforms and more general self-maps of the complex upper half plane, these semi-groups are in bijection with certain classes of functions defined through their analytic and asymptotic properties.  This bijection provides a  L\`{e}vy-Hin\u{c}in  representation for these infinitely divisible distributions.
In section \ref{combinatorics} we prove the main result of the paper, namely the extension of Theorem \ref{scalartheorem} to the operator-valued case.  In contrast to the previous section, this is achieved through a combinatorial methodology.  We close the paper with the Appendix, which is primarily concerned with the extension of the main result in \cite{Williams-Analytic}, namely the classification of the Cauchy transforms associated to $\mathcal{B}$-valued distributions, to a more general class of functions including the Cauchy transforms associated to more general CP maps.

\section{Preliminaries}\label{prelim}
Let $\mathcal{B}$ denote a unital C$^{\ast}$-algebra and $X$ a self-adjoint symbol.  We will define the ring of noncommutative polynomials $\mathcal{B} \langle X \rangle$ as the algebraic free product of $\mathcal{B}$ and $X$. $\mathcal{B}_0 \langle X \rangle$ are polynomials in $\mathcal{B} \langle X \rangle$ with zero constant term.

\begin{Defn}
Let $\mu: \mathcal{B} \langle X \rangle \rightarrow \mathcal{B}$ denote a linear map.
We say that $\mu$ is \emph{exponentially bounded} with constant $M$  if
\begin{equation}\label{bound}
\| \mu(b_{1}Xb_{2} \cdots X b_{n+1} ) \| \leq M^{n} \|b_{1} \|  \|b_{2} \|  \cdots \|b_{n+ 1} \|
\end{equation}
We abuse terminology and say that the map $\mu$ is \textit{completely positive} (CP) if
\begin{equation}\label{CP}
(\mu \otimes 1_{n}) \left( \left[ P_{i}(X)P_{j}^{\ast}(X) \right]_{i,j = 1}^{n} \right) \geq 0
\end{equation}
for every family $P_{i}(X) \in \mathcal{B}\langle X \rangle$.

We define a set $\Sigma_{0}$ to be those $\mathcal{B}$-bimodular linear maps $\mu$ satisfying \eqref{bound} and \eqref{CP}.
\end{Defn}

For a general introduction to non-commutative functions, we refer to \cite{KVV}.
Throughout, $\mathcal{B} , \mathcal{A}$ shall denote unital C$^{\ast}$-algebras.
Let $M_{n}(\mathcal{B})$ denote the $n\times n$ matrices with entries in $\mathcal{B}$.
We define the \textit{noncommutative space over $\mathcal{B}$} to be the set $\mathcal{B}_{nc} = \{ M_{n}(\mathcal{B}) \}_{n=1}^{\infty}$.
A \textit{non-commutative set} is a subset $\Omega \subset \mathcal{B}_{nc}$
that respects direct sums.  That is, for $X \in \Omega \cap M_{n}(\mathcal{B})$ and $Y \in \Omega \cap M_{p}(\mathcal{B})$ we have that $X \oplus Y \in \Omega \cap M_{n+p}(\mathcal{B})$.
We note that these definitions apply to the more general case of $\mathcal{B}$ being any unital, commutative ring,  but we focus on the $C^{\ast}$-algebraic setting.  Given $b \in M_{n}(\mathcal{B})$, the \textit{non-commutative ball} of radius $\delta$ about $b$ is the set
$B_{\delta}^{nc}(b):= \sqcup_{k=1}^{\infty} B_{\delta}(\oplus^{k} b)$ where $ B_{\delta}(\oplus^{k} b) \subset M_{nk}(\mathcal{B})$ is the standard ball of radius $\delta$.

A \textit{non-commutative} function is a map $f : \Omega \rightarrow \mathcal{A}_{nc}$ with the following properties:
\begin{enumerate}
 \item $f(\Omega_{n}) \subset M_{n}(\mathcal{A})$
\item  $f$ respects direct sums : $f(X \oplus Y) = f(X) \oplus f(Y)$
\item $f$ respects similarities: For $X \in \Omega_{n}$ and $S \in M_{n}(\mathbb{C})$ invertible we have that
$$f(SXS^{-1}) = Sf(X)S^{-1} $$
provided that $SXS^{-1} \in \Omega_{n}$.
\end{enumerate}

A non-commutative function is said to be \textit{locally bounded in slices} if, for every $n$ and element $x \in \Omega_{n}$,
$f|_{\Omega_{n}}$ is bounded on some neighborhood of $x$ in the norm topology.  It is a remarkable fact originally due to Taylor (\cite{Taylor1}, \cite{Taylor2}) that a
non-commutative function that is G\^{a}teaux differentiable and locally bounded in slices is in fact analytic.
A non-commutative function is \textit{uniformly analytic} at $b \in M_{n}(\mathcal{B})$ if it is analytic and bounded on $B_{r}^{nc}(b)$ for some $r > 0$.

Let $M_{n}^{+,\epsilon}(\mathcal{B}) \subset M_{n}(\mathcal{B})$ denote those element $b\in M_{n}(\mathcal{B})$ with
$\Im{(b)} > \epsilon 1_{n}$ and $M_{n}^{+}(\mathcal{B}) = \cup_{\epsilon > 0} M_{n}^{+,\epsilon} $.
We form a non-commutative set
$$H^{+}(\mathcal{B}) = \sqcup_{n=1}^{\infty} M_{n}^{+}(\mathcal{B}) $$
and refer to this set as the \textit{non-commutative upper half plane}.

We define a family of sets in $H^{+}(\mathcal{B})$.
For $\alpha, \epsilon  > 0$ define a \textit{ non-commutative Stolz angle } to be
$$ \Gamma_{\alpha, \epsilon}^{(n)} := \{ b \in M_{n}^{+,\epsilon}(\mathcal{B}) : \Im{(b)} > \alpha \Re{(b)} \}. $$

Let $\mu \in \Sigma_{0}$.  We define the \textit{Cauchy transform} of $\mu$ to be the analytic, non-commutative function $G_{\mu} = \{G_{\mu}^{(n)} \}_{n=1}^{\infty}$ such that
$$ G_{\mu}^{(n)}(b):= (\mu \otimes 1_{n}) ((b-X\otimes 1_{n})^{-1}) : H^{+}(\mathcal{B}) \mapsto H^{-}(\mathcal{B}).$$
From this map, we may construct the \textit{moment generating function}, the \textit{F-transform}, the \textit{Voiculescu transform} and the $\mathcal{R}$-\textit{transform} respectively through the following equalities:
$$H^{(n)}(b) :=  G^{(n)}(b^{-1}) : H^{-}(\mathcal{B}) \mapsto H^{-}(\mathcal{B}) $$
$$ F^{(n)}(b) := G^{(n)}(b)^{-1} : H^{+}(\mathcal{B}) \mapsto H^{+}(\mathcal{B})$$
$$ \varphi_{\mu}^{(n)}(b) :=( F_{\mu}^{(n)})^{\langle -1 \rangle}(b) - b $$
$$ \mathcal{R}_{\mu}^{(n)}(b) := \varphi_{\mu}^{(n)}(b^{-1})$$
where the superscript $\langle -1 \rangle$ refers to the composition inverse.  We also note that the moment generating function extend to a neighborhood of $0$ for $\mu \in \Sigma_{0}$ and that the Voiculescu-transform is only defined on a subset of $H^{+}(\mathcal{B})$.
The following result, proven in \cite{Williams-Analytic} and \cite{Popa-Vinnikov-NC-functions}, classifies the $F$-transforms in terms of their analytic and asymptotic properties.

\begin{theorem}\label{nevanlinna}
Let $f = (f^{(n)}): H^{+}(\mathcal{B}) \rightarrow H^{+}(\mathcal{B})$ denote an analytic, noncommutative function.
The following conditions are equivalent.
\begin{enumerate}[(a)]
\item\label{nev1}  $f = F_{\mu}$ for some $\mu \in \Sigma_{0}$.
\item\label{nev2}  The noncommutative function $k = (k^{(n)})_{n=1}^{\infty}$ defined by $k^{(n)}(b) := (f^{(n)}(b^{-1}))^{-1}$ has
uniformly analytic extension to a neighborhood of $0$. Moreover, for any sequence $\{ b_{k} \}_{k \in \mathbb{N}}$ with $\|b_{k}^{-1} \| \downarrow 0$, $b_{k}^{-1}f^{(n)}(b_{k}) \rightarrow 1_{n}$ in norm.
\item\label{nev3} There exists an $\alpha \in \mathcal{B}$ and a $\sigma : \mathcal{B}\langle X \rangle \rightarrow \mathcal{B}$ which satisfies \eqref{bound} and \eqref{CP} such that, for all $n \in \mathbb{N}$,
$$f^{(n)}(b)  = \alpha 1_{n} + b - (\sigma \otimes 1_{n}) (b(1-Xb)^{-1}).$$
\end{enumerate}
Moreover, the map $\sigma$ in \eqref{nev3} is of the form $\sigma(P(X)) = \rho(XP(X)X)$ for $\rho$ such that its restriction to $\mathcal{B}_{0}\langle X \rangle$ is positive.
\end{theorem}

We will require several classical results in complex function theory to prove our results.
 Theorem 3.16.3 in \cite{Hille-Phillips} is a useful analogue of the classical
Cauchy estimates in complex analysis.  We also refer to this reference for an overview of the differential structure of vector valued functions, including the higher order derivative $\delta^{n}$ utilized below.
\begin{theorem}\label{Cauchy}
Let $f$ be G\^{a}teaux differentiable in $\mathcal{U}$ and assume that $\|f(x) \| \leq M$ for $x \in \mathcal{U}$.
Then $$\| \delta^{n}f(a;h) \| \leq Mn!$$
for $a + h \in \mathcal{U}$.
\end{theorem}
Further, theorem 3.17.17 in \cite{Hille-Phillips} provides Lipschitz estimates for analytic functions.
Indeed, for an analytic function $f$ that is locally bounded by $M(a)$ in a neighborhood of radius $r_{a}$, we have that
\begin{equation}\label{lipschitz}
\| f(y) - f(x) \| \leq \frac{2M(a) \| x - y \|}{r_{a} - 2 \| x - y \|}
\end{equation}

\begin{Notation}
We define a family $\Lambda$ of functions $\Phi : H^{+}(\mathcal{B}) \rightarrow \overline{ H^{-}(\mathcal{B}) }$
through the following properties:
\begin{enumerate}[(i)]
\item\label{omega1} The map $\mathcal{R}(b) := \Phi(b^{-1})$ has uniformly analytic continuation to a non-commutative ball about $0$ with $\mathcal{R}(b)^{\ast} = \mathcal{R}(b)$
\item\label{omega2}  For any sequence $\{ b_{k} \}_{k \in \mathbb{N}} \in \mathcal{B}$ with $\| b_{k} ^{-1}\| \downarrow 0$, we have that $b_{k}^{-1} \Phi(b_{k}) \rightarrow 0$.
\end{enumerate}

We also define a larger family of functions $\tilde{\Lambda}$ by replacing \eqref{omega1} and \eqref{omega2} with the following weaker conditions
\begin{enumerate}[(I)]
\item For any $\epsilon > 0$, $\Phi$ is uniformly bounded on $\sqcup_{n=1}^{\infty} M_{n}^{+,\epsilon}(\mathcal{B})$.
\item For any $\alpha , \epsilon > 0$ and a sequence $\{ b_{k} \}_{k \in \mathbb{N}} \in \Gamma^{(n)}_{\alpha , \epsilon} $ with $\| b_{k} ^{-1}\| \downarrow 0$, we have that $b_{k}^{-1} \Phi(b_{k}) \rightarrow 0$.
\end{enumerate}
\end{Notation}

\begin{Defn}
\label{Defn:Monotone-convolution}
Let $\mu , \nu \in \Sigma_{0}$.
We define the \textit{monotone convolution} to be the non-commutative operation $ (\mu , \nu) \mapsto  \mu \rhd \nu \in \Sigma_{0}$
defined implicitly though the equality
$$ F_{\mu \rhd \nu} := F_{\mu} \circ F_{\nu}.$$
Note that this definition uses Theorem~\ref{nevanlinna} in an essential way, to show that a composition of $F$-transforms is an $F$-transform. See Section~\ref{combinatorics} and references \cite{Popa-Monotonic-Cstar,Hasebe-Saigo-Monotone-cumulants,Popa-Conditionally-monotone,Hasebe-Saigo-Operator-monotone} for the relation between this definition and monotone independence of Muraki.
\end{Defn}

\begin{Defn}
\label{Defn:infinitely-divisible}
We say that $\mu$ is a  $\rhd$-infinitely divisible distribution if, for every $n$, there exists a distribution $\mu_{n} \in \Sigma_{0}$ such that
\begin{equation}\label{IDdef}
\mu = \underbrace{\mu_{n} \rhd \mu_{n} \rhd \cdots \rhd \mu_{n}}_{n \ times}
\end{equation}

We define a composition semigroup of $F$-transforms $\{ F_{t} \}_{t \in \mathbb{Q}^{+}}$ by letting
$F_{p/q} := F_{\mu_{q}}^{\circ  p}$ where $\mu = \mu_{q}^{\rhd q}$ for all $p,q \in \mathbb{N}$.
We will show in Theorem \ref{generation} that this semigroup extends to an $\mathbb{R}^{+}$ semigroup, which moreover is generated by a function $\Phi \in \Lambda$ in a sense that will be made specific.  Moreover, one of the main results in \cite{Williams-Analytic} is that the set $\Lambda$ is exactly the set of Voiculescu transforms associated to $\boxplus$-infinitely divisible distributions.  This is not a coincidence and will drive the main result of this paper.
\end{Defn}

\section{L\'{e}vy-Hin\u{c}in Representations for Semigroups of Non-Commutative Functions.}\label{analysis}

We begin this section with a result showing that the divisors of $\rhd$-infinitely divisible distributions maintain the same exponential bound.  A similar result can be proven in the combinatorial setting of Section~\ref{combinatorics} in an easier manner, but the bound is less sharp.

\begin{proposition}\label{exp}
Let $\mu$ denote a $\rhd$-infinitely divisible distribution with exponential bound $M$.  Then, for each $k$, the distribution $\mu_{k}$ satisfying
$\mu = \mu_{k}^{\rhd k}$ has exponential bound $M$.
\end{proposition}
\begin{proof}
Let $Xb_{1}Xb_{2} \cdots b_{n-1} X = Q(X) \in \mathcal{B}\langle X \rangle$ such that $\|b_{1} \| = \|b_{2} \| = \cdots \| b_{n-1} \| = 1$ and assume, for the sake of contradiction, that $\| \mu_{k}(Q(X))\|  > M^{n}$.  Then, using the Schwarz inequality for $2$-positive maps, we have that
\begin{align*}
\| \mu_{k}(Q^{\ast}(X)Q(X) )\| \| \mu_{k}(1) \|  & \geq \| \mu_{k}(Q(X)) \mu_{k}(Q^{\ast}(X) ) \| \\ &
 = \| \mu_{k}(Q(X)) \|^{2}  > M^{2n}
\end{align*}
Since $\mu_{k}(1) = 1$, we may assume that our monomial  $P(X) = X b_{1}Xb_{2} \cdots b_{n-1} X^{2}  b_{n -1}^{\ast}  X \cdots b_{1}^{\ast}X$ has the property that $\mu_{k}(P(X)) > M^{2n}$.
Define an element $B \in M_{2n}(\mathcal{B})$ by
$$ B = \left( \begin{array} {ccccccccc}
0 &1 & 0 & 0 & 0 & 0 & 0 & \cdots & 0 \\
1 & 0 & b_{1} & 0 & 0 & 0 & 0 & \cdots & 0 \\
0 & b_{1}^{\ast} & 0 & 1 & 0 & 0 & 0 & \cdots & 0 \\
0 & 0 & 1 & 0 & b_{2}  & 0 & 0 & \cdots & 0 \\
0 & 0 & 0 & b_{2}^{\ast} & 0  & 1 & 0 & \cdots & 0 \\
0 & 0 & 0 & 0 & 1 & 0 & b_{3}  & \cdots & 0 \\
\ & \vdots & \ & \ & \vdots & \ & \ & \ & \vdots \\
0 & 0 & 0 & 0 &   0 & \cdots & b^{\ast}_{n-1} & 0 &1\\
0 & 0 & 0 & 0 & 0 & \cdots & 0  & 1 & 0
\end{array} \right) .$$
That is, the superdiagonal alternates between $1$ and $b_{i}$, the subdiagonal alternates between  $1$ and $b_{i}^{\ast}$.
Now, let $0 < \epsilon , \delta$ and
$$ B_{\delta , \epsilon} = \delta B + \epsilon \left( \sum_{i=1}^{2n -1} e_{i,i} \right) + \frac{e_{2n  , 2n}}{\delta^{n -1}}$$
where $\epsilon$ is arbitrarily small and $\delta$ is chosen so that $B_{\delta , \epsilon}$ is a strictly positive element.
Moreover, we have that
\begin{align}\label{powereq}
e_{1,1} ( B_{\delta , \epsilon}  (X \otimes 1_{2n}) B_{\delta, \epsilon})^{2n}e_{1,1} & = e_{1,1}  B_{\delta , \epsilon}  [(X \otimes 1_{2n}) B_{\delta, \epsilon}^{2}]^{2n - 1}  (X \otimes 1_{2n}) B_{\delta, \epsilon} e_{1,1} \\  & = P(X) + O(\max{(\delta,\epsilon} )). \nonumber \end{align}
To see this, note that a non-trivial contribution to  \eqref{powereq} must be of the form
$$ b_{1,2} X b_{2,j_{3}} b_{j_{3}, j_{4}} X b_{j_{4} , j_{5}} X \cdots b_{j_{4n - 2} , j_{4n - 1} } b_{j_{4n - 1}, 2} X b_{2,1}  $$
where $b_{i,j}$ denotes the $i,j$ entry of $B_{\delta , \epsilon}$.
Now, such a non-zero term is \textit{not} $O(\max{(\delta,\epsilon} ))$
means that $b_{j_{\ell} ,j_{\ell + 1}}$ must equal $b_{2n,2n}$ for two distinct $\ell$.  However, the only possible way for this to occur is if $j_{k} = k $ for $k = 2 , \ldots , 2n $, $j_{2n} = j_{2n + 1} = j_{2n +2} = 2n$ and $j_{p} = 4n + 2- p $ for $p = 2n + 2 , \ldots , 4n - 1$.

By assumption, there exists a state $\phi \in \mathcal{B}^{\ast}$ such that $\phi(\mu_k(P(X))) > M^{2k}$. Thus, for $\epsilon$ small enough, we have that
\begin{equation}
\phi_{1,1} \circ (\mu_{k}\otimes 1_{2n}) ((  B_{\delta , \epsilon} (X \otimes 1_{2n}) B_{\delta, \epsilon})^{2n}) > M^{2n}
\end{equation}
(here $\phi \otimes e_{1,1} = \phi_{1,1}$).
 This implies that  the scalar valued Cauchy transform associated to this random variable,
$$G_{\mu_{k}}^{\delta, \epsilon}(z) = \phi_{1,1} \circ (\mu_{k}\otimes 1_{2n}) ((z1_{2n} - B_{\delta,\epsilon}  (X \otimes 1_{2n})B_{\delta, \epsilon})^{-1})$$
arises from a measure whose support has non-trivial intersection with $\mathbb{R} \setminus [-M,M]$, whereas the (similarly defined)
$G_{\mu}^{\delta , \epsilon}$ has support contained in $[-M, M]$ (since its moments have growth rate smaller than powers of $M$).
Using Stieltjes inversion, this implies that
\begin{equation}\label{continue}
\lim_{t \downarrow 0} -\Im{ G_{\mu_{k}}^{\delta, \epsilon}(x + it)} > 0
\end{equation}
for some $x > M$ (or the limit simply does not exist in the atomic case).

Calculating the imaginary part of this Cauchy transform, we have
\begin{align}
\Im( [\mu_{k} ((z1_{2n} - B_{\delta, \epsilon}XB_{\delta, \epsilon})^{-1})]^{-1} ) & = B_{\delta, \epsilon}^{-1} \Im([\mu_{k}(B_{\delta, \epsilon}^{-2}z - X)^{-1}]^{-1})B_{\delta, \epsilon}^{-1} \nonumber \\
& =  B_{\delta, \epsilon}^{-1} \Im F^{(n)}_{\mu_{k}}(zB_{\delta, \epsilon}^{-2})   B_{\delta, \epsilon}^{-1} \nonumber \\
& \leq B_{\delta, \epsilon}^{-1} \Im F^{(n)}_{\mu}(zB_{\delta, \epsilon}^{-2})   B_{\delta, \epsilon}^{-1} \nonumber \\ \label{imag}
& = \Im( [\mu ((z1_{2n} - B_{\delta, \epsilon}XB_{\delta, \epsilon})^{-1})]^{-1}  )
\end{align}
where the inequality follows from the fact that $F_{\mu} = F_{\mu_{k}}^{\circ k-1} \circ F_{\mu_{k}}$ and $F$-transforms increase the imaginary part.

Rewriting the right hand side of \eqref{imag}, we have that
\begin{align}
\Im( &  [\mu ((z1_{2n} - B_{\delta, \epsilon}XB_{\delta, \epsilon})^{-1})]^{-1}) \nonumber \\  \nonumber
& = [\mu ((z1_{2n} - B_{\delta, \epsilon}XB_{\delta, \epsilon})^{-1}) ^{\ast}]^{-1}
\Im( \mu ((z1_{2n} - B_{\delta, \epsilon}XB_{\delta, \epsilon})^{-1}))
  [\mu ((z1_{2n} - B_{\delta, \epsilon}XB_{\delta, \epsilon})^{-1})]^{-1} \\
&= F_{\mu}^{\delta, \epsilon}(z)^{\ast} \Im{ ( F_{\mu}^{\delta, \epsilon}(z) )}  F_{\mu}^{\delta, \epsilon}(z)
\end{align}
We conclude that
\begin{equation}\label{Ftraneq}
\Im( [\mu ((z1_{2n} - B_{\delta, \epsilon}XB_{\delta, \epsilon})^{-1})]^{-1}) \leq
F_{\mu}^{\delta, \epsilon}(z)^{\ast} \Im{ ( F_{\mu}^{\delta, \epsilon}(z) )} F_{\mu}^{\delta, \epsilon}(z).
\end{equation}
Since $ F_{\mu}^{\delta, \epsilon}$ extends to $\mathbb{R} \setminus [-M,M]$
$$ \lim_{t \downarrow 0}  G_{\mu}^{\delta, \epsilon}(x + it)$$
converges to a positive element in $\mathcal{B}$ and
$$ \lim_{t \downarrow 0} \Im{ ( F_{\mu}^{\delta, \epsilon}(x + it) )} \rightarrow 0$$
it follows that the right hand side or \eqref{Ftraneq} converges to $0$ in norm, contradicting \eqref{continue}.  This completes our proof.
\end{proof}

\begin{proposition}\label{convergence}
Let $\mu, \mu_k$ be as in the preceding proposition.
We have that $F_{\mu_{k}} \rightarrow Id$ in norm as $k\uparrow \infty$ uniformly on $M_{n}^{+,\epsilon}(\mathcal{B})$, and this convergence is also uniform over $n$ .
Moreover,
the functions $F^{(n)}_{\mu_{k}}(b^{-1}) - b^{-1}$ and $F^{(n)}_{\mu_{k}}(b^{-1}) ^{-1}$
extend analytically to $B^{nc}_{r}(0)$, where the radius $r$ is dependent only on $M$ from Proposition~\ref{exp},  and satisfy
\begin{equation}\label{nbhdinf1}
F^{(n)}_{\mu_{k}}(b^{-1}) - b^{-1} \rightarrow 0_{n}
\end{equation}
\begin{equation}\label{nbhdinf2}
F^{(n)}_{\mu_{k}}(b^{-1}) ^{-1} = H^{(n)}_{\mu_{k}}(b) \rightarrow b
\end{equation}
where this convergence is uniform  on $B^{nc}_{r}(0)$.
\end{proposition}
\begin{proof}
Consider the Nevanlinna representations of each of these functions
\begin{equation}\label{revision}  F_{\mu_{k} }^{(n)}(b) = \alpha_{k} \otimes 1_{n} + b - G_{\rho_{k}}^{(n)}(b)  \end{equation}
defined in  Theorem~\ref{nevanlinna} , where we have adopted the notation that $\mu = \mu_{1}$.  We claim that the distributions  $\rho_{k}$ share a common exponential bound $N$ for all $k \in \mathbb{N}$.

To prove this claim, first observe that, by Theorem 4.1 in \cite{Williams-Analytic}, there exist distributions  $\nu_{k}$ such that
$$ b - F_{\mu_{k}}^{(n)}(b) = \varphi^{(n)}_{\nu_{k}}(b) = - \alpha_{k} \otimes 1_{n} + G_{\rho_{k}}^{(n)}(b).$$
Moreover, it was shown in \cite{Popa-Vinnikov-NC-functions} that if the $\nu$ and the $\nu_{k}$ have a common exponential bound $N$ then the distributions $\rho$ and $\rho_{k}$ have a common exponential bound $N^{2} + 1$.  Focusing on the $\nu_{k}$, we may manipulate equations \ref{revision} to conclude that
\begin{equation}\label{RTrans}
\mathcal{R}_{\nu_{k}}(b^{-1} ) = \varphi_{\nu_{k}}(b) = b^{-1} - F_{\mu_{k}}(b^{-1}).
\end{equation}

Now, expand the moment series
\begin{equation}\label{momentexpansion}
F^{(n)}_{\mu_{k}}(b^{-1})^{-1}  =  H^{(n)}_{\mu_{k}}(b) = \sum_{p=0}^{\infty} \mu_{k}((bX)^{p}b). \end{equation}
Note that Proposition~\ref{exp} implies that this function is convergent and uniformly bounded for $b \in B^{nc}_{r}(0)$, independent of $k$.

Observe that the moment generating function satisfies
\begin{equation}\label{momentexpansion2} [ H^{(n)}_{\mu_{k}}(b)]^{-1} = b^{-1} - \mu_{k}(X) + \mu_{k}(X)b\mu_{k}(X) - \mu_{k}(XbX) + \cdots = b^{-1} + f^{(n)}(b,X)\end{equation}
where $ f^{(n)}(b,X)$ is analytic in $b$ and converges for $\|b \|$ small, where the radius of convergence is only dependent on $M$.
Thus, $[ H^{(n)}_{\mu_{k}}(b)]^{-1} - b^{-1}$ extends to a neighborhood of $0$ whose radius is independent of $n$ and $k$ and agrees with $F^{(n)}_{\mu_{k}}(b^{-1}) - b^{-1}$ when $b$ is invertible.
Moreover, these observations, combined with \eqref{RTrans} imply that the functions $\mathcal{R}_{\nu_{k}}$ have a common $R , C > 0$ such that the functions extend to a common domain
$B^{nc}_{R}(0)$ with a common bound $C$.  Now a careful look at the Kantorovich argument in part \textrm{II} of the proof of Theorem $4.1$ in \cite{Williams-Analytic} allows us to conclude that the exponential bound on the distributions $\nu_{k}$ depend only on $R$, proving our claim.

Recall that $F_{\mu_{k}} \circ \cdots \circ F_{\mu_{k}} = F_{\mu}$
we have that
\begin{equation}
G_{\rho}^{(n)}(b) =  G_{\rho_{k}}^{(n)}(b) +  G_{\rho_{k}}^{(n)} \circ F_{\mu_{k} }^{(n)}(b)+ \cdots +  G_{\rho_{k}}^{(n)} \circ \underbrace{ F_{\mu_{k} }^{(n)}\circ \cdots \circ F_{\mu_{k} }^{(n)}}_{k - 1 \ times} (b)
\end{equation}
Letting $b = z1_{n}$ for $z \in \mathbb{C}$, we have that
\begin{align*}
\lim_{|z| \uparrow \infty} zH_{\rho}^{(n)} \left(\frac{1}{z}1_{n} \right) & = \lim_{|z| \uparrow \infty} zG_{\rho}^{(n)}(z1_{n})  \\
& = \lim_{|z| \uparrow \infty}  \sum_{\ell = 1}^{k-1} z G_{\rho_{k}}^{(n)} \circ (F_{\mu_{k} }^{(n)})^{\circ \ell} (z1_{n}) \\
& = \lim_{|z| \uparrow \infty} \sum_{\ell = 1}^{k-1} z H_{\rho_{k}}^{(n)} \left( [(F_{\mu_{k} }^{(n)})^{\circ \ell} (z1_{n})]^{-1} \right) \\
& = \lim_{|z| \uparrow \infty} \sum_{\ell = 1}^{k-1} z H_{\rho_{k}}^{(n)} \circ G_{\nu_{\ell}}^{(n)}\left( z 1_{n} \right) \\
& = \lim_{|w| \downarrow 0} \sum_{\ell = 1}^{k-1} \frac{1}{w} H_{\rho_{k}}^{(n)} \circ G_{\nu_{\ell}}^{(n)}\left( \frac{1}{w} 1_{n} \right) \\
& = \lim_{|w| \downarrow 0} \sum_{\ell = 1}^{k-1} \frac{1}{w} H_{\rho_{k}}^{(n)} \circ H_{\nu_{\ell}}^{(n)}\left(w1_{n} \right)
\end{align*}
where $[(F_{\mu_{k} }^{(n)})^{\circ \ell}]^{-1} = G_{\nu_{\ell}} $ is the Cauchy transform of a distribution $\nu_{\ell} \in \Sigma_{0}$ (this follows from  Theorem~\ref{nevanlinna}).
Moreover, we have that
$$\lim_{|w| \downarrow 0}  \frac{1}{w} H_{\nu_{\ell}}^{(n)}\left(w1_{n} \right) = 1_{n} $$
so that, passing to limits and utilizing the chain rule and the fact that  $H_{\nu_{\ell}}^{(n)}(0_{n}) = 0_{n}$ , we have that
 $$ \delta H_{\rho}^{(n)} (0_{n} ; 1_{n}) =  k \delta H_{\rho_{k}}^{(n)} (0_{n} ; 1_{n}) $$
Utilizing the main result in our appendix, Theorem~\ref{generalization}, we conclude that
\begin{equation}
\rho(1) = \mu(X^{2}) = k\mu_{k}(X^{2}) = k \rho_{k}(1).
\end{equation}
so that $\rho_{k}(1) = O(1/k)$.

Now, assume that $b \in M_{n}^{+,\epsilon}(\mathcal{B})$.  We claim that $\| b^{-1} \| \leq 1/\epsilon$.
Indeed, observe that, for $b = x + iy$ with $y > \epsilon 1_{n}$,
\begin{equation}\label{binverse}
b = \sqrt{y}(i + (\sqrt{y})^{-1} x (\sqrt{y})^{-1} ) \sqrt{y}
\end{equation}
(it follows easily from this equation that $b$ is invertible, but this is known).  Thus,
\begin{equation}\label{binverse2}
b^{-1} = (\sqrt{y})^{-1}(i + (\sqrt{y})^{-1} x (\sqrt{y})^{-1} )^{-1} (\sqrt{y})^{-1}.
\end{equation}
Now, utilizing  the spectral mapping theorem and the fact that the spectral radius agrees with the norm for normal operators, we have that  $\| (\sqrt{y})^{-1} \| \leq (\sqrt{\epsilon})^{-1}$.  Moreover, since $i + (\sqrt{y})^{-1} x (\sqrt{y})^{-1} $
is normal and has spectrum with imaginary part larger than $1$, we have that
$(i + (\sqrt{y})^{-1} x (\sqrt{y})^{-1})^{-1}$ is normal and, by the same spectral considerations, has norm bounded by $1$.  These observations, combined with \eqref{binverse} imply our claim.

Thus, for $b \in M_{n}^{+}(\mathcal{B})$, we have
\begin{align*}
\| F^{(n)}_{\mu_{k}}(b) - b \| & \leq \| \alpha_{k} \| + \| (\rho_{k}\otimes 1_{n})((b-X)^{-1} \| \\
& \leq \| \alpha \|/k + \| (b-X)^{-1} \| \|(\rho_{k}\otimes 1_{n}) (1_{n}) \| \\
& \leq  \frac{\| \alpha \|}{k} +  \frac{\|\rho_{k}(1) \|}{\epsilon} = \frac{\| \alpha \| + \rho(1) / \epsilon }{k}
\end{align*}
and the right hand side converges to zero uniformly over $M_{n}^{+,\epsilon}(\mathcal{B})$, independent of $n$.

Regarding the second part of our Proposition, we first observe that each of the moments  of $\mu_k$ converges to $0$.
  Indeed, utilizing the Schwarz inequality for $2$-positive maps as well as Proposition~\ref{exp}, we have that
\begin{align*}
 \| \mu_{k} (X b_{1} X b_{2}X \cdots b_{\ell}X) \|^{2}  & \leq \| \mu_{k}(X^{2}) \| \|  \mu_{k}(Xb_{\ell}^{\ast} X \cdots b_{2}^{\ast}X b_{1}^{\ast} b_{1} X b_{2} X\cdots  b_{\ell}X) \| \\
& \leq  \frac{\| \mu(X^{2} ) \| M^{2\ell} \|b_{1} \|^{2} \| b_{2}\|^{2} \cdots \| b_{\ell} \|^{2}}{k}
\end{align*}

Moreover, the tail of the series expansion of  $f^{(n)}(b,X) $  is bounded in norm independent of $n$ and $k$ .
 the individual entries all go to $0$ so the we conclude that  $f^{(n)}(b,X) \rightarrow 0 $ uniformly on $b \in B^{nc}_{r}(0)$ as $k\uparrow \infty$ so that we can immediately conclude that \eqref{nbhdinf2} holds.
This completes our proof.
\end{proof}

We next prove a differentiation result for vector valued functions.  We adapt a proof found in \cite{BP78} of a similar result for complex functions.

\begin{proposition}\label{differentiate0}
Let $\mathcal{A}$ and $\mathcal{B}$ denote unital Banach algebras.  Consider an open subset $\Omega \subset \mathcal{A}$.
Let $f_{t} : \Omega \mapsto \mathcal{B}$ for all $t \geq 0$ be a composition semigroup of analytic functions.
Assume that for every $b' \in \Omega$, there exists a $\delta > 0$ such that
\begin{enumerate}[(a)]
\item\label{Banach1}  $ \lim_{t \downarrow 0} f_{t}(b) - b \rightarrow 0$ uniformly over $b \in B_{\delta}(b')$
\item\label{Banach2}  For any $T > 0$, we have that $f_{t}(b) - b$ is uniformly bounded over $b \in B_{\delta}(b')$   and $t \in  [0,T]$.
\end{enumerate}
Then, there exists an analytic $\Phi : \Omega \mapsto \mathcal{B}$ such that
\begin{equation}\label{diff}
\frac{d f_{t}(b)}{dt} = - \Phi(f_{t}(b)).
\end{equation}
\end{proposition}
\begin{proof}


Fix $b' \in \Omega$.
We first claim that there exists an $\alpha> 0$ such that
\begin{equation}\label{eq2}
\| f_{2t}(b) - 2f_{t}(b) + b \| \leq \frac{1}{10}\| f_{t}(b) - b \|.
\end{equation}
for all $t \in [0,\alpha]$ and $b \in B_{\delta/2}(b')$ where the value of $\delta$ comes from the statement .

Indeed, fix  $b \in B_{\delta/2}(b')$. We first consider the simple case when there exists a sequence $t_n \downarrow 0$ such that $f_{t_n}(b) = b$. Since $\set{f_t}$ form a composition semigroup, this property then holds for a dense set of $t$'s, and by continuity assumption in part (a), for all $t > 0$. So \eqref{eq2} holds trivially.

Thus, suppose that $f_t(b) \neq b$ for $t \in [0, \alpha]$. Define a family of complex functions $g_{t}$ through the following equalities:
$$ h_{t} := \frac{f_{t}(b) - b}{\| f_{t}(b) - b\|} ; \ g_{t}(\zeta) := f_{t}(b + \zeta h_{t}) - b : B_{\delta/2}(0) \mapsto \mathcal{B}.$$
where $B_{\delta/2}(0)$ refers to the neighborhood of zero in the complex plane. Note that, since we are taking a ball of radius $\delta/2$, we may define $h_{t}$ for all such $b$ provided that our choice of $\alpha$ is small enough.

Consider the vector valued complex integral
\begin{equation}\label{integrand} \int_{0}^{\| f_{t}(b) - b \|} \frac{d}{d\zeta} [g_{t}(\zeta) - \zeta h_{t}] d\zeta .\end{equation}
By \eqref{Banach1} and the Cauchy estimates in Theorem~\ref{Cauchy}, the integrand can be made arbitrarily small for $t$ small.  By the fundamental theorem, this integral is equal to
$$ g_{t}(\| f_{t}(b) - b \|)  - g_{t}(0) - (f_t(b) - b) = f_{t}( b + (f_{t}(b) - b)) - b - 2 (f_{t}(b) - b) = f_{2t}(b) - 2f_{t}(b) + b.$$
Using our bound on the integrand, equation \eqref{eq2} follows immediately.

We now use \eqref{eq2} to prove that for $\alpha > 0$  there exists an $M > 0$ such that
\begin{equation}\label{eq3}
\| f_{t}(b) - b \| \leq Mt^{2/3}
\end{equation}
for all $t \in [0,\alpha] $ and  $b \in B_{\delta/2}(b')$.
Indeed, pick $t \in [0 , \alpha] $ and $m \in \mathbb{N}$ such that $2^{m}t \leq \alpha < 2^{m+1}t$.
Note that inequality \eqref{eq2} and the triangle inequality imply that
$$2\|f_{t}(b) - b \| - \| f_{2t}(b) - b \|  \leq \|f_{2t}(b) - 2f_{t}(b) + b \| \leq \frac{1}{10}\| f_{t}(b) - b \| $$
so that
\begin{equation}
\| f_{t}(b) - b \| \leq \frac{10}{19}  \|f_{2t}(b) - b \|   \leq 2^{-2/3} \| f_{2t}(b) - b \|
\end{equation}
Using this estimate inductively, we have
$$ \| f_{t}(b) - b \| \leq  2^{-2/3} \| f_{2t}(b) - b \|  \leq \cdots \leq 2^{-2m/3} \| f_{2^{m}t}(b) - b \|
= t^{2/3} \left( \frac{1}{2^{m}t}\right)^{2/3} M'$$
where $M'$ is a  bound on $\|f_{s}(b) - b \| $ for $s \leq 2$ which exists by \eqref{Banach2}.
Equation \eqref{eq3} follows with $M = 2^{2/3}M'/\alpha$.

Now, revisiting the argument for \eqref{eq2}, inequality \eqref{eq3} implies that the integrand in \eqref{integrand} has bound equal to  $$2M t^{2/3} $$
as a result of the Cauchy estimates.  Thus,  we have the following:
\begin{equation}\label{eq4}
\| f_{2t}(b) - 2f_{t}(b) + b \| \leq 2t^{2/3} \|f_{t}(b) - b\| \leq 2 M t^{4/3}.
\end{equation}
We may further conclude that
\begin{equation}\label{eq5}
\left\| \frac{f_{2t}(b) - b}{2t} - \frac{f_{t}(b) - b}{t}  \right\|  \leq Mt^{1/3}
\end{equation}
Thus, we have  that \begin{equation}\label{limit} \lim_{k \uparrow \infty} 2^{k}(f_{2^{-k}}(b) - b) \end{equation}
converges uniformly on $ B_{\delta/2}(b')$ and we refer to this limit as $-\Phi(b)$.

Using \eqref{eq5}, we note that $\Phi$ is locally bounded.  Indeed, we have that
\begin{align}\nonumber
\|  2^{p}(f_{1/2^{p}}(b) - b) + \Phi(b) \| \leq & \sum_{k=p}^{\infty} \| 2^{k}(f_{1/2^{k}}(b) - b) - 2^{k + 1}(f_{1/2^{k + 1}}(b) - b) \| \\ \label{eq6}
&  \leq  \frac{M}{2}\sum_{k=p}^{\infty} \left( \frac{1}{2^{1/3} } \right)^{k} =MC(p).
\end{align}
for all $b\in B_{\delta/2}(b')$.  Local boundedness of $\Phi$ follows since $(f_{1/2^{p}}(b) - b)$ is locally bounded.  Also note that $C(p) \rightarrow 0$ as $p \uparrow \infty$.

 Regarding analyticity of $\Phi$, consider a state $\varphi \in \mathcal{B}^{\ast}$ , $b \in B_{\delta/2}(b')$, and an element $h \in \mathcal{B}$ with $\| h \| \leq 1$.
We define  complex maps $$H_{m}(z): B_{\delta/2}(0) \subset \mathbb{C} \rightarrow \mathbb{C}$$
for $m \geq 0$  through the equalities:
$$ H_{0}(z) := \varphi \circ \Phi ( b + zh) ; \ \ H_{m}(z) := 2^{m} \varphi \circ (f_{ 2^{-m}}(b + zh) - (b + zh)). $$
By \eqref{limit}, $H_{m} \rightarrow H_{0}$  for $z \in B_{\delta/2}(0)$, and by \eqref{eq6}, the limit is bounded on this set.
Thus, $H_{0}$ is analytic in $z$.  By Dunford's theorem (\cite{Dun}), it follows that
$ \Phi(b + zh) $ is analytic in $z$ and, therefore, G\^{a}teaux differentiable.
As this function is locally bounded, it is analytic.

Regarding \eqref{diff},
observe that $\{ f_{t}(b) \}_{t \geq 0}$ is compact since it is the continuous image of $[0,t]$.
As \eqref{Banach1} and \eqref{Banach2} hold on neighborhoods of every point in this set, taking a finite cover, we have that \eqref{Banach1} and \eqref{Banach2} holds uniformly on a neighborhood of this set and, after a close look at the relevant constants, \eqref{eq6} is also maintained on this set.
Now, fix $t \geq 0$ and let $\ell_{p}/2^{p} \rightarrow t$ as $p \uparrow \infty$.
\begin{align*}
f_{t}(b) - b & = (f_{t}(b) - f_{t - \ell_{p}/2^{p}}(b)) + \sum_{j = 1}^{\ell_{p}} (f_{j/2^{p}}(b) - f_{(j-1)/2^{p}}(b)) \\
&= (f_{t}(b) - f_{t - \ell_{p}/2^{p}}(b)) + \sum_{j = 1}^{\ell_{p}} \frac{1}{2^{p}}(2^{p}[ f_{j/2^{p}}(b) - f_{(j-1)/2^{p}}(b) ] )
\end{align*}
As $p \uparrow \infty$, $$f_{t}(b) - f_{t - \ell_{p}/2^{p}}(b) = f_{\ell_{p}/2^{p}} \circ f_{t - \ell_{p}/2^{p}}(b)  -  f_{t - \ell_{p}/2^{p}}(b) \rightarrow 0$$ since \eqref{Banach1} holds on the entire path.  Moreover, the remaining summand is simply a Riemann sum approximation of a sequence of functions converging uniformly to $-\Phi \circ f_{s}(b)$ for $s \in [0,t]$.
The following equation follows immediately:
$$f_{t}(b) = b - \int_{0}^{t} \Phi \circ f_{s}(b) ds. $$
We conclude that \eqref{diff} holds, completing our proof.
\end{proof}

\begin{corollary}\label{differentiate}
Let $\mathcal{A}$ and $\mathcal{B}$ denote Banach algebras and $\Omega \subset \sqcup_{n=1}^{\infty} M_{n}(\mathcal{A})$
a non-commutative set.
Let $F_{t} : \Omega \mapsto \sqcup_{n=1}^{\infty} M_{n}(\mathcal{B})$  for all $t\geq 0$ and assume that they form a composition semigroup of analytic non-commutative functions.  Assume that, for each $n$, the composition semigroup of vector valued analytic functions $\{ F_{t}^{(n)} \}_{t \geq 0}$ satisfies the hypotheses of Proposition~\ref{differentiate0}.
Then there exists an analytic, noncommutative map $\Phi :\Omega \mapsto \sqcup_{n=1}^{\infty} M_{n}(\mathcal{B})$
such that
\begin{equation}\label{diff2}
\frac{d F^{(n)}_{t}(b)}{dt} = - \Phi^{(n)}(F^{(n)}_{t}(b))
\end{equation}
for all $n\in \mathbb{N}$, $b \in \Omega_{n}$.

Moreover, if we strengthen these assumptions so that, for any $n$ and  $b \in M_{n}(\mathcal{B})$, there exists a $\delta > 0$ with
\begin{enumerate}
\item\label{Banach21}  $ \lim_{t \downarrow 0} F_{t} - Id \rightarrow 0$ uniformly over $ B_{\delta}^{nc}(b)$.
\item\label{Banach22}  For any $T > 0$, we have that $f_{t}(b) - b$ is uniformly bounded on $ B_{\delta}^{nc}(b)$  and $t \in  [0,T]$.
\end{enumerate}
then $\Phi$ is uniformly analytic.
\end{corollary}
\begin{proof}
We showed in Proposition~\ref{differentiate0} this map $\Phi$ exists.  We must show that it is a non-commutative function.
However, this is immediate since, for $b_{1} \in M_{n}(\mathcal{B})$ and $b_{2} \in M_{p}(\mathcal{B})$, we have
\begin{align*}
\Phi^{(n+p)}(b_{1} \oplus b_{2}) & = \lim_{k \uparrow \infty} 2^{k}(F^{(n+p)}_{2^{-k}}(b_{1} \oplus b_{2}) - b_{1} \oplus b_{2} )\\ & = \lim_{k \uparrow \infty} 2^{k}( [F^{(n)}_{2^{-k}}(b_{1} ) - b_{1}] \oplus [ F^{(n)}_{2^{-k}}(b_{2}) -  b_{2}]) \\
&= \Phi^{(n)}(b_{1}) \oplus  \Phi^{(p)}(b_{2}) .
\end{align*}
A similar proof shows that it also satisfies the defining invariance property so that our first claim holds.

With respect to the uniform analyticity, we refer to the proof of Proposition~\ref{differentiate0}.  Observe that inequality \eqref{eq2} holds for $\alpha$ small enough.  This $\alpha$ is only
dependent on the convergence of the integrand in \eqref{integrand}.  This converges to $0$ uniformly on $B_{\delta}^{nc}(b)$ by assumption \eqref{Banach21} and the same Cauchy estimate
so that the choice of $\alpha$ is also uniform on this set.  Moreover, the constant $M$ in \eqref{eq3} is equal to $2^{2/3}M'/\alpha$
where $M'$ is the upper bound on $F_{s} - Id$ for $s \leq \alpha$.  Assumption \eqref{Banach22} implies that this bound is uniform on
$B_{\delta}^{nc}(b)$.  Thus, inequality \eqref{eq6} holds on all of this set, implying uniform analyticity.
\end{proof}

\begin{theorem}\label{generation}
Let $\{ F_{t} \}_{t \in \mathbb{Q}^+}$ denote a composition semigroup of non-commutative functions $F_{t}: H^{+}(\mathcal{B}) \mapsto  H^{+}(\mathcal{B})$ such that
\begin{enumerate}[(i)]
\item\label{Ftran1} $\|F_{t}^{(n)}(b) - b \| \rightarrow 0$ uniformly on $M_{n}^{+,\epsilon}(\mathcal{B})$ for all $\epsilon > 0$, independent of $n$ as $t \downarrow 0$.
\item\label{Ftran2}  For any $\alpha , \epsilon > 0$ and sequence $b_{k} \in \Gamma^{(n)}_{\alpha , \epsilon} $ with $\| b_{k} ^{-1}\| \downarrow 0$, we have that $b_{k}^{-1} F_{t}^{(n)}(b_{k}) \rightarrow 1_{n}$ as $k \uparrow \infty$
\item\label{Ftran3} $\Im{F^{(n)}_{t}(b)} \geq \Im{b}$ for all $b \in M_{n}^{+}(\mathcal{B})$ and $t \geq 0$.
\end{enumerate}
Then $\{ F_{t} \}_{t \in \mathbb{Q}^+}$ extends to a semigroup $\{ F_{t} \}_{t \geq 0}$ and the map $\Phi$ from Proposition~\ref{differentiate} is an element of $\tilde{\Lambda}$.

Since, by Proposition~\ref{convergence}, the conditions above are satisfied by $F$-transforms, this implies that a $\rhd$-infinitely divisible distribution $\mu$ as in Definition~\ref{Defn:infinitely-divisible} can be realized as $\mu = \mu_1$ for a monotone convolution semigroup $\set{\mu_t}_{t \geq 0}$. For such a semigroup, $\Phi \in \Lambda$.

Conversely, given a map $\Phi \in \tilde{\Lambda}$ we may construct a semigroup of non-commutative functions satisfying the hypotheses above as well as the differential equation
\begin{equation}\label{differentiate2}
\frac{d F_{t}(b)}{dt} = -\Phi(F_{t}(b))
\end{equation}
  If $\Phi \in \Lambda$ then the semigroup arises from a $\rhd$-infinitely divisible distribution.
\end{theorem}

We shall refer to this element $\Phi$ as the \textit{generator} or the semigroup $\{ F_{t} \}_{t \geq 0}$.
\begin{proof}

 First, let $\Phi \in \tilde{\Lambda}$.
We will produce the semigroup it generates by the method of successive approximations.

Consider a sequence of non-commutative functions $\{ f_{k}(t,\cdot) \}_{t \geq 0 , \ k \in \mathbb{N}}$ defined as follows:
\begin{equation}\label{successive}
 f^{(n)}_{1}(t,b) = b ; \ \ f^{(n)}_{k+1}(t,b) = b - \int_{0}^{t} \Phi(  f^{(n)}_{k}(s,b)) ds.
\end{equation}
We claim that
$f_{k}(t,\cdot)$ is convergent and satisfies the semigroup property with generator $\Phi$.

Observe that since $\Phi$ is uniformly bounded by a constant $M$ on set $M_{n}^{+,\epsilon/2}(\mathcal{B})$
and $f_{k}(t,\cdot)$ maps the set $M_{n}^{+,\epsilon}(\mathcal{B})$ to itself since $$\Phi : H^{+}(\mathcal{B}) \mapsto  H^{-}(\mathcal{B})$$ we have that \begin{equation}\label{imaginary}\Im{f^{(n)}_{k}(t,b)}\geq \Im{(b)}.\end{equation}
By \eqref{lipschitz}, this implies that  $f^{(n)}_{k}(t,\cdot)$ is Lipschitz on the set $B_{\epsilon/2}(b) \subset M_{n}^{+,\epsilon/2}(\mathcal{B})$ for all $b \in  M_{n}^{+,\epsilon}(\mathcal{B})$, and the Lipschitz constant $L$ is uniform over both $k$, $b$ and bounded $t$.  Moreover, we may extend the Lipschitz inequality
$$ \| f_{k}(t,b) - f_{k}(t,b') \| \leq L \| b - b' \| $$
to all $b , b' \in M_{n}^{+,\epsilon}(\mathcal{B})$ by taking a path $b + s(b'-b)$ for $s \in [0,1]$ and using the Lipschitz estimate on intervals of distance $\epsilon/2$ since the distances are additive on this path.
Using this Lipschitz estimate in the integrand of \eqref{successive},
we conclude that
\begin{equation}
\|  f^{(n)}_{2}(t,b) -  f^{(n)}_{1}(t,b) \| = t \| \Phi(b) \| \leq tML
\end{equation}
and we may conclude that
\begin{align*}
\|  f^{(n)}_{3}(t,b) -  f^{(n)}_{2}(t,b) \| & = \left\| \int_{0}^{t}  [\Phi( f^{(n)}_{2}(s,b)) - \Phi( f^{(n)}_{1}(s,b))]ds \right\| \\
& \leq L \left\| \int_{0}^{t} [ f^{(n)}_{2}(s,b) -  f^{(n)}_{1}(s,b) ] ds \right\|  \\
& \leq L \int_{0}^{t} [ LMs ] ds  \leq \frac{t^{2}L^{2}M}{2}
\end{align*}
Continuing inductively, we have that
\begin{equation}\label{boundsum}
\|  f^{(n)}_{k+1}(t,b) -  f^{( n)}_{k}(t,b) \| \leq \frac{M(Lt)^{k+1}}{L(k+1)!}.
\end{equation}
For any choice of $t \in [0 , \alpha]$, we have that
\begin{equation}\label{boundsum2}
f^{(n)}_{N+1}(t, b) - b =  \sum_{k=0}^{N} \left( f^{(n)}_{k+1}(t,b) -  f^{( n)}_{k}(t,b) \right)
\end{equation}
is a convergent series as $N \uparrow \infty$ and we may conclude that  $f_{N}(t, \cdot ) $ converges to a function $  f (t,\cdot )$ uniformly on $M_{n}^{+,\epsilon}(\mathcal{B})$, independent of $n$.

It is clear that $f(t,\cdot)$ satisfies \eqref{differentiate2}.
Regarding the asymptotics, let $\alpha , \epsilon > 0$ and fix a sequence $b_{\ell} \in \Gamma^{(n)}_{\alpha , \epsilon} $ with $\| b_{\ell} ^{-1}\| \downarrow 0$.
Note that $b^{-1}_{\ell}f_{1}^{(n)}(t,b_{\ell}) \equiv 1_{n}$ and satisfies $\|f_{1}^{(n)}(t,b_{\ell})  \|^{-1} \downarrow 0$ as $\| b^{-1}_{\ell} \| \downarrow 0$.
We claim $b^{-1}_{\ell}f_{k}^{(n)}(t,b_{\ell}) \rightarrow 1_{n}$ and satisfies $\|f_{k}^{(n)}(t,b_{\ell})  \|^{-1} \downarrow 0$  as $\| b^{-1}_{\ell} \| \downarrow 0$ for all $k$,  uniformly over $t \in [0,\alpha]$.

Proceeding by induction, we have that for fixed $k$
\begin{equation}\label{convergeto1} b_{\ell}^{-1} f^{(n)}_{k+1}(t,b_{\ell}) = 1_{n} -   \int_{0}^{t} [ b_{\ell}^{-1}f^{(n)}_{k}(s,b_{\ell})] (f^{(n)}_{k}(s,b_{\ell}))^{-1}\Phi(  f^{(n)}_{k}(s,b_{\ell})) ds.\end{equation}
We bound the integrand by
$$ \| [ b_{\ell}^{-1}f^{(n)}_{k}(s,b_{\ell})]\| \| (f^{(n)}_{k}(s,b_{\ell}))^{-1}\Phi(  f^{(n)}_{k}(s,b_{\ell}))\|$$
which converges to $0$  uniformly over $s \in [0,\alpha]$ by induction, so that \eqref{convergeto1} converges to $1_{n}$.
Moreover,
$$ \| [f^{(n)}_{k+1}(t,b_{\ell})]^{-1} \| \leq \| b_{\ell}^{-1} \| \| b_{\ell} [f^{(n)}_{k+1}(t,b_{\ell})]^{-1} \| \rightarrow 0.$$
Thus, each $f_{k}(t,\cdot)$ has the appropriate asymptotics and, since $f(t,\cdot)$ is a uniform limit of these functions on $M_{n}^{+,\epsilon}$, our claim holds  Condition (iii) follows from \eqref{imaginary}.

In order to complete our proof, we further assume that $\Phi \in \Lambda$ and prove that the functions $f(t, \cdot)$ are in fact the
$F$-transforms of noncommutative distributions $\mu_{t} \in \Sigma_{0}$.  To do so we must show that
the function $f(t,b^{-1})^{-1}$ has a uniformly analytic extension to a neighborhood of $0$ for all $t \geq 0$.
Note that, since $\Phi \in \Lambda$, there exists a $\delta > 0$ and constants $M , L > 0$ such that $\Phi^{(n)}(b^{-1})$
extends to $B^{nc}_{\delta}(0) $ with upper bound $M$ and Lipschitz constant $L$.

Now fix $\alpha >0$. We claim that, for $\gamma > 0 $ small enough we have that $f_{k}^{(n)}(t,b^{-1})^{-1}$ extends to $B_{\gamma}(0_{n}) \subset M_{n}(\mathcal{B})$ for all $n$ and satisfies $f_{k}^{(n)}(t,b^{-1})^{-1} \in B_{\delta}(0_{n})$ for all $b \in B_{\gamma}(0_{n}) $.
Choose any $t \in [0,\alpha]$ and $b \in B_{\gamma}(0_{n}) $ where $\gamma < \delta$ is yet unspecified.  We have
\begin{align*}
\| f^{(n)}_{2}(t , b^{-1})^{-1} -  f^{(n)}_{1}(t , b^{-1})^{-1} \| & = \left\|  \left[ \left(1_{n} - \int^{t}_{0} b\Phi(b^{-1}) ds \right)^{-1} - 1_{n} \right]b  \right\|  \\
& \leq \sum_{n=1}^{\infty}  \left\| \int^{t}_{0} b\Phi(b^{-1}) ds \right\|^{n} \|b \| \\
& \leq \gamma \sum_{n=1}^{\infty} (\gamma M\alpha)^{n}\\
& = \frac{\gamma^{2}M \alpha }{1 - \gamma M \alpha}
\end{align*}
Deriving a similar inequality for general $k$, we have that
\begin{equation}
\begin{split}
\|& f^{(n)}_{k+1}(t , b^{-1})^{-1} -  f^{(n)}_{k}(t , b^{-1})^{-1} \| \label{cauchydiff} \\
&= \left\| \left( b^{-1} - \int_{0}^{t} \Phi \circ f^{(n)}_{k}(s , b^{-1})ds  \right)^{-1} - \left(  b^{-1} - \int_{0}^{t} \Phi \circ f^{(n)}_{k-1}(s , b^{-1})ds  \right)^{-1}  \right\|  \\
&= \left\| \left( 1_{n} -\int^{t}_{0} b\Phi(f^{(n)}_{k}(t , b^{-1}) ) \right)^{-1}  \left(  b\int^{t}_{0} \Phi(f^{(n)}_{k-1}(t , b^{-1}) ) - \Phi(f^{(n)}_{k}(t , b^{-1}) ) \right) \right. \\
&\qquad \left. \left( 1_{n} -\int^{t}_{0} b\Phi(f^{(n)}_{k-1}(t , b^{-1}) ) \right)^{-1}b \right\|  \\
& \leq \left( \frac{1}{1 - \gamma M \alpha}\right)^{2}  (\gamma^{2} L \alpha) \|  f^{(n)}_{k}(t , b^{-1})^{-1} -  f^{(n)}_{k-1}(t , b^{-1})^{-1}\| \\
\end{split}
\end{equation}
By induction, we have that
$$ \| f^{(n)}_{k+1}(t , b^{-1})^{-1}  - b \| =  \sum_{\ell =1}^{k} \frac{M \gamma^{2\ell}L^{\ell-1} \alpha^{\ell} }{(1 - \gamma M \alpha)^{2\ell - 1} }$$
This is convergent as $k \uparrow \infty$ for $\gamma$ small and converges to $0$ as $\gamma \downarrow 0$.  Thus, for $\gamma$ small enough,
we have that  $f^{(n)}_{k+1}(t , b^{-1}) \in B_{\delta}(0_{n})$ for all $k$ and $n$  and, therefore,  converges to a limit function on $B_{\gamma}(0_{n})$  (since the differences in \eqref{cauchydiff} are Cauchy).
This limit function must agree with $f(t,\cdot)$ by analytic continuation.   This completes our proof that  $f(t,\cdot)$ is an $F$-transform for all $t$.



To address the converse, consider a semigroup $\{F_{t} \}_{t \in \mathbb{Q}^{+}}$ satisfying the \eqref{Ftran1} and \eqref{Ftran2} in the statement of the theorem. First note that this easily extends to an $\mathbb{R}^{+}$ composition semigroup.
Indeed, define $F_{t}(b) = \lim_{p/q \rightarrow t} F_{p/q}(b)$.
To see that this is well defined, note that, as $p/q , p'/q' \rightarrow t$, we have $$\|F^{(n)}_{p/q}(b) - F^{(n)}_{p'/q'}(b) \|
= \|F^{(n)}_{p/q - p'/q'} \circ  F^{(n)}_{p'/q'}(b) - F^{(n)}_{p'/q'}(b) \| \rightarrow 0$$
uniformly on $M_{n}^{+,\epsilon}(\mathcal{B})$
by property \eqref{Ftran1} and \eqref{Ftran3} . It is immediate that this is a composition semigroup over $\mathbb{R}^{+}$ satisfying \eqref{Ftran1}, \eqref{Ftran2} and \eqref{Ftran3}.

By  Corollary \ref{differentiate}, this semigroup may be differentiated to produce a non-commutative function $\Phi$ .
Regarding the asymptotics of $\Phi$, consider the inequality
\begin{equation}\label{asymptote}
 \|b^{-1} \Phi^{(n)}(b) \| \leq \left\| \frac{b^{-1}(F^{(n)}_{t}(b) - b)}{t}  \right\| + \|b^{-1}\| \left\|\frac{ (F^{(n)}_{t}(b) - b)}{t} - \Phi^{(n)}(b) \right\| .
\end{equation}
Utilizing inequality \eqref{eq6} in the proof of Proposition \ref{differentiate0} produces
\begin{equation}\label{uniformity}
 \left\|\frac{ (F^{(n)}_{2^{N}}(b) - b)}{2^{N}} - \Phi^{(n)}(b) \right\| \leq M\sum_{k=N +1}^{\infty}\left( \frac{1}{2^{1/3} } \right)^{k}
\end{equation}
where this $M = 2M'/\alpha$ .  As was noted in the proof of Corollary \ref{differentiate}, uniform convergence in the sense of \eqref{Ftran1} and \eqref{Ftran2} implies a uniform bound on $M$.  Thus, \eqref{uniformity} converges to $0$ uniformly on
$M_{n}^{+,\epsilon}(\mathcal{B})$ so that, for fixed $t$ small enough, second term on the right hand side of \eqref{asymptote} is smaller than any $\delta > 0$ for $b \in M_{n}^{+,\epsilon}(\mathcal{B})$.  Letting $b_{k} \in \Gamma^{(n)}_{\alpha , \epsilon} $ satisfy
$\| b_{k} ^{-1}\| \downarrow 0$, the first term on the right hand side of \eqref{asymptote} converges to $0$ by assumption \eqref{Ftran2}, and it follows that $\Phi \in \tilde{\Lambda}$.

If $\{F_{t} \}_{t \geq 0}$ arises from a $\rhd$-infinitely divisible measure, then it follows from Proposition~\ref{exp} and Theorem~\ref{nevanlinna} that $b_{k}^{-1} F_{\mu_{t}}^{(n)}(b_{k}) \rightarrow 1_{n}$ for any sequence $b_{k} \in M_{n}(\mathcal{B})$ with $\| b_{k} ^{-1}\| \downarrow 0$ and a similar proof allows one to conclude that $\Phi$ satisfies condition \eqref{omega2} in the definition of $\Lambda$.

It remains to show that $\Phi$ satisfies \eqref{omega1}.
However, Proposition \ref{convergence} implies that
there exists a fixed $r > 0$ such that each function $F_{\mu_{t}}^{(n)}(b^{-1}) - b^{-1}$ extends to $B_{r}(\{ 0\})$ and converges to $0$ uniformly on this set.
Thus, the strengthened hypotheses in Corollary \ref{differentiate0} hold so that the non-commutative function defined by the equalities
$$\mathcal{R}^{(n)}(b) = \lim_{t\downarrow 0} \frac{ F^{(n)}_{\mu_{t}}(b^{-1}) - b^{-1}}{t} $$
is uniformly analytic at  $0$ and, by continuation, is an extension of $\Phi^{(n)}(b^{-1})$ for each $n$.
Thus, $\Phi \in \Lambda$, completing our proof.
\end{proof}

The following proposition establishes continuity in generating the semigroups, and may be useful in future applications.

\begin{proposition}
Assume that $\Phi_{1} , \Phi_{2} \in \tilde{\Lambda}$ generate the semigroups of noncommutative functions $\{ F_{1}(t,\cdot) \}_{t \geq 0}$ and $\{ F_{2}(t,\cdot) \}_{t \geq 0}$.
If we assume that $\| \Phi^{(n)}_{1}(b) - \Phi^{(n)}_{2}(b) \| < \epsilon$ for all $b \in B_{\delta}(b') \subset M_{n}(\mathcal{B})$, a ball of radius $\delta$ where $\Im{(b')} > \delta1_{n}$,
then we may conclude that $\| F^{(n)}_{1}(1,b) - F^{(n)}_{2}(1,b)\| < C\epsilon$  for all $b \in B_{\delta}(b')$  where $C$ depends only on $\Phi_{1}$.
\end{proposition}

\begin{proof}
To prove our claim, we first note that, by the vector-valued chain rule,
$$ \frac{\delta^{2}F^{(n)}_{i}(t,b)}{\delta t^{2}} =  \delta \Phi^{(n)} \left( F^{(n)}_{i}(t,b) , \frac{\delta}{\delta t} F^{(n)}(b,t)  \right)$$
so that $F_{i}(t,b)$ is twice differentiable in $t$ and has uniformly bounded derivative for $b \in H^{+,\epsilon}(\mathcal{B})$ and $t \in [0,1]$.
We refer to the maximum of this bound over $i=1,2$ as $M_{2}$.

Using the remainder estimates for the Taylor series associated to $F_{i}$, we have the following:
\begin{equation}\label{gamma}
\| F_{i}(b , t + \gamma) -  F_{i}(b , t)  - \gamma \Phi(F_{i}(b,t)) \| \leq \frac{M_{2} \gamma^2}{2}
\end{equation}

Let $M_{1} = \sup_{b \in M_{n}^{+,\epsilon}(\mathcal{B}) , \ n \in \mathbb{N} } \| \delta \Phi^{(n)}(b, \cdot) \|$.
Utilizing the estimate \eqref{gamma} with $\gamma = 1/N$, we produce the following inequalities:
\begin{align*}
 \| &  F_{1}^{(n)}(b, t_{0} + 1/N)  -  F_{2}^{(n)}(b, t_{0} + 1/N)\| \\
&  \leq \frac{M_{2}}{N^{2}}  + \frac{1}{N} \| \Phi_{1}^{(n)}(F_{1}^{(n)}(b,t_{0})) - \Phi_{2}^{(n)}(F_{2}^{(n)}(b,t_{0})) \| + \| F_{1}^{(n)}(b,t_{0}) - F_{2}^{(n)}(b,t_{0})\| \\
& \leq \frac{M_{2}}{N^{2}} + \frac{1}{N} \| \Phi_{1}^{(n)}(F_{1}^{(n)}(b,t_{0})) - \Phi_{1}^{(n)}(F_{2}^{(n)}(b,t_{0})) \| \\
& \quad + \frac{1}{N} \| \Phi_{1}^{(n)}(F_{2}^{(n)}(b,t_{0})) - \Phi_{2}^{(n)}(F_{2}^{(n)}(b,t_{0})) \|  +  \| F_{1}^{(n)}(b,t_{0}) - F_{2}^{(n)}(b,t_{0})\| \\
& \leq \frac{M_{2}}{N^{2}} + \frac{\epsilon}{N} + \left(1 +  \frac{M_{1}}{N} \right)\| F_{1}^{(n)}(b,t_{0}) - F_{2}^{(n)}(b,t_{0})\|
\end{align*}
Using this estimate inductively, we have that
\begin{align*}
 \| &  F_{1}^{(n)}(b,1)  -  F_{2}^{(n)}(b,1)\| \leq  \left(  \frac{\epsilon}{N} + \frac{M_{2}}{N^{2}} \right) \sum_{k=0}^{N-1}  \left( 1 + \frac{M_{1}}{N} \right)^{k} \rightarrow \frac{e^{M_{1}} - 1}{M_{1}} \epsilon
\end{align*}
where the convergence occurs as $N \uparrow \infty$.  This implies our result.

\end{proof}

\section{The Bercovici-Pata Bijection.}\label{combinatorics}

\begin{Defn}
Let $(S, \prec)$ be a poset. An \emph{order} on $S$ is an order-preserving bijection
\[
f : (S, \prec) \rightarrow \left(\set{1, 2, \ldots, \abs{S}}, < \right).
\]
Denote by $o(S)$ the number of different orders on $S$.
\end{Defn}

\begin{Lemma}
\label{Lemma:Order}
Let $(S, \prec)$ be a poset, and $S = U \sqcup V$ a partition of $S$. $U$ and $V$ are posets with the induced order.
\begin{enumerate}
\item
Suppose for all $u \in U$ and $v \in V$, $u \prec v$. Then
\[
o(S) = o(U) o(V).
\]
\item
Suppose that all $u \in U$ and $v \in V$, $u$ and $v$ are unrelated to each other. Then
\[
\frac{o(S)}{\abs{S}!} = \frac{o(U)}{\abs{U}!} \frac{o(V)}{\abs{V}!}.
\]
\end{enumerate}
\end{Lemma}

\begin{proof}
Part (a) is obvious. It is also clear that there is a bijection between the orders on $S$ and triples
\[
\set{\text{order on $U$, order on $V$, a subset of $\set{1, 2, \ldots, \abs{S}}$ of cardinality $\abs{U}$}}.
\]
Therefore
\[
o(S) = \binom{\abs{S}}{\abs{U}} o(U) o(V).
\]
This implies part (b).
\end{proof}

\begin{Defn}
For a non-crossing partition $\pi = \set{V_1, V_2, \ldots, V_k}$, define a partial order on it as follows: for $U, V \in \pi$,
\[
U \prec V \text{ if } \exists i, j \in U \ \forall v \in V : i < v < j.
\]
In this case we say that $U$ \emph{covers} $V$. Minimal elements with respect to this order are called the \emph{outer} blocks of $\pi$; the rest are the \emph{inner} blocks.
\end{Defn}

See \cite{Hasebe-Saigo-Monotone-cumulants,Hasebe-Saigo-Operator-monotone} for more on orders on non-crossing partitions.

\begin{Defn}
Let $\mu : \mc{B} \langle X \rangle \rightarrow \mc{B}$ be a $\mc{B}$-bimodule map; at this point no positivity assumptions are made. Its \emph{monotone cumulant functional} is the $\mc{B}$-bimodule map $K^\mu : \mc{B}_0 \langle X \rangle \rightarrow \mc{B}$ defined implicitly by
\[
\mu[b_0 X  b_1 X \ldots b_{n-1} X b_n]
= \sum_{\pi \in \NC(n)} \frac{o(\pi)}{\abs{\pi}!} K^\mu_\pi[b_0 X  b_1 X \ldots b_{n-1} X b_n].
\]
Here $K_\pi$ is defined in the usual way as in \cite{SpeHab}, see Section~3 of \cite{Ans-Bel-Fev-Nica} for a detailed discussion.
\end{Defn}

\begin{Remark}
For $n \in \mf{N}$, we note that
\[
K^{1_n \otimes \mu} = 1_n \otimes K^\mu.
\]
The proof of this fact is identical to that of Proposition~6.3 of \cite{Popa-Vinnikov-NC-functions}.

It follows that the generating function arguments in the rest of this section work equally well for each $1_n \otimes \mu$, and so the corresponding generating functions completely determine the states.
\end{Remark}

\begin{Lemma}\label{MntCnt}
\label{Lemma:Continuity}
For $\mc{B}$-bimodule maps, $\mu_i \rightarrow \mu$ if and only if $K^{\mu_i} \rightarrow K^\mu$.
\end{Lemma}

\begin{proof}
By assumption, $\mu_n[b] = b = \mu[b]$. For $n \geq 1$, one implication is clear, and the other follows by induction on $n$.
\end{proof}

\begin{Defn}
For $\mu$ as above and $\eta : \mc{B} \rightarrow \mc{B}$ a linear map, define $\mu^{\rhd \eta}$ via
\[
K^{\mu^{\rhd \eta}}[b_0 X  b_1 X \ldots b_{n-1} X b_n] = b_0 \eta \left(K^\mu[X  b_1 X \ldots b_{n-1} X]\right) b_n.
\]
\end{Defn}


Define the formal generating functions
\[
H^\mu(b) = \sum_{n=0}^\infty \mu[b (X b)^n]
\]
and
\[
K^\mu(b) = \sum_{n=1}^\infty K^\mu[b (X b)^n].
\]
Note that as formal series,
\[
H^\mu(b) = G^\mu(b^{-1}),
\]
so our notation is consistent with the analytic function notation in the rest of the article, except that we use superscripts for formal series. Note also that these generating functions differ by a factor of $b$ from the more standard ones, and are more appropriate for the computations with monotone convolution.

The following results may be contained in \cite{Popa-Monotonic-Cstar}, and are closely related to Proposition~3.5 in \cite{Hasebe-Saigo-Operator-monotone}. We provide a purely combinatorial direct proof.

\begin{Prop}
\label{Prop:Evolution}
Let $\mu : \mc{B} \langle X \rangle \rightarrow \mc{B}$ be an exponentially bounded $\mc{B}$-bimodule map. Then for each $d$
\[
\frac{d H^{(1_d \otimes \mu)^{\rhd t}}(b)}{d t}  = K^{1_d \otimes \mu}(H^{(1_d \otimes \mu)^{\rhd t}}(b)).
\]
\end{Prop}

\begin{proof}
It suffices to prove the result for $d=1$. We begin by proving this equality for each of the coefficients of the series expansions of
$H^{\mu^{\rhd t}}$ and $K^\mu \circ H^{\mu^{\rhd t}}$.

Towards this end, fix $n \in \mathbb{N}$ and  $\pi \in \NC(n)$. Denote by $V_1, \ldots, V_k$ the outer blocks of $\pi$, by $c(V_i)$ the partition consisting of $V_i$ and the inner blocks it covers, and by $c_j(V_i)$, $j = 1, 2, \ldots, \abs{V_i} - 1$ the partition consisting of the inner blocks lying between the $j$th and the $(j+1)$st elements of $V_i$. By Lemma~\ref{Lemma:Order} part (b),
\[
\frac{o(\pi)}{\abs{\pi}!} = \prod_{i=1}^k \frac{o(c(V_i))}{\abs{c(V_i)}!}.
\]
By part (a) of the lemma,
\[
o(c(V_i)) = o \left( \bigcup_{j=1}^{\abs{V_i} - 1} c_j(V_i) \right)
\]
and so by part (b),
\[
\frac{o(c(V_i))}{(\abs{c(V_i)} - 1)!} = \prod_{j=1}^{\abs{V_i} - 1} \frac{o(c_j(V_i))}{\abs{c_j(V_i)}!}.
\]
Since
\[
\begin{split}
\frac{d}{d t} \mu^{\rhd t}[b (X b)^n]
& = \frac{d}{d t} \sum_{\pi \in \NC(n)} t^{\abs{\pi}} \frac{o(\pi)}{\abs{\pi}!} K^\mu_\pi[b X b X \ldots b X b] \\
& = \sum_{\pi \in \NC(n)} t^{\abs{\pi}-1} \frac{o(\pi)}{(\abs{\pi} - 1)!} K^\mu_\pi[b X b X \ldots b X b],
\end{split}
\]
the coefficient of $K^\mu_\pi[b (X b)^n]$ in its expansion is $t^{\abs{\pi}-1} \frac{o(\pi)}{(\abs{\pi} - 1)!}$. On the other hand,
\[
\begin{split}
& K^\mu \left[H^{\mu^{\rhd t}}(b) \left(X H^{\mu^{\rhd t}}(b) \right)^l \right] \\
& = K^\mu \left[H^{\mu^{\rhd t}}(b) X H^{\mu^{\rhd t}}(b) X \ldots H^{\mu^{\rhd t}}(b) X H^{\mu^{\rhd t}}(b) \right] \\
& = \sum_{k_0, \ldots, k_l \geq 0} K^\mu \left[\sum_{\pi_0 \in \NC(k_0)} t^{\abs{\pi_0}} \frac{o(\pi_0)}{\abs{\pi_0}!} K^\mu_{\pi_0} X \sum_{\pi_1 \in \NC(k_1)} t^{\abs{\pi_1}} \frac{o(\pi_1)}{\abs{\pi_1}!} K^\mu_{\pi_1} X \ldots X \sum_{\pi_l \in \NC(k_l)} t^{\abs{\pi_l}} \frac{o(\pi_l)}{\abs{\pi_l}!} K^\mu_{\pi_l} \right] \\
& = \sum_{k_0, \ldots, k_l \geq 0} \sum_{\substack{\pi_0 \in \NC(k_0), \\ \vdots \\ \pi_l \in \NC(k_l)}} \frac{o(\pi_0)}{\abs{\pi_0}!} \frac{o(\pi_1)}{\abs{\pi_1}!} \ldots \frac{o(\pi_l)}{\abs{\pi_l}!}
K^\mu \left[ K^\mu_{\pi_0} X  K^\mu_{\pi_1} X \ldots X  K^\mu_{\pi_l} \right] t^{\abs{\pi_0} + \abs{\pi_1} + \ldots + \abs{\pi_l}},
\end{split}
\]
where $K_\emptyset(b) = b$. Fixing $n = k_0 + \ldots + k_l + l$, each term in this expansion is a multiple of $K^\mu_\pi[b (X b)^n]$, where $\pi$ is constructed from partitions $\pi_0, \pi_1, \ldots, \pi_k$ and an additional outer block of $l$ elements:
\[
V = \set{k_0 + 1, k_0 + k_1 + 2, \ldots, k_0 + \ldots + k_{l-1} + l} \in \pi
\]
and
\[
\pi_i = \text{ restriction of } \pi \text{ to } [k_0 + \ldots + k_{i-1} + i + 1, k_0 + \ldots + k_i + i], \quad i = 0, 1, \ldots, l.
\]
Note that $\abs{\pi_0} + \abs{\pi_1} + \ldots + \abs{\pi_l} = \abs{\pi} - 1$. This identification has an inverse, which requires first choosing one of the $k$ outer blocks of $\pi$. Therefore the coefficient of $K^\mu_\pi[b (X b)^n]$ in the expansion of $K^\mu(H^{\mu^{\rhd t}}(b))$ is $t^{\abs{\pi}-1}$ times
\[
\begin{split}
\sum_{i=1}^k \frac{o \left( \bigcup_{j < i} c(V_j) \right)}{\abs{\bigcup_{j < i} c(V_j)}!} \left( \prod_{j=1}^{\abs{V_i} - 1} \frac{o(c_j(V_i))}{\abs{c_j(V_i)}!} \right) \frac{o \left( \bigcup_{j > i} c(V_j) \right)}{\abs{\bigcup_{j > i} c(V_j)}!}
& = \sum_{i=1}^k \frac{o(c(V_i))}{(\abs{c(V_i)} - 1)!} \prod_{j \neq i} \frac{o(c(V_j))}{\abs{c(V_j)}!} \\
& = \prod_{i=1}^k \frac{o(c(V_i))}{\abs{c(V_i)}!} \sum_{i=1}^k \abs{c(V_i)} \\
& = \abs{\pi} \prod_{i=1}^k \frac{o(c(V_i))}{\abs{c(V_i)}!}.
\end{split}
\]
Since this is the same coefficient as in the first expansion, the result is proved for each of the individual components of the respective series expansions for each $n \in \mathbb{N}$.

Extending this to the series expansions and, therefore, the functions, observe that all of the sets over which the sums occur have cardinality whose growth rate is exponential over $n$.  Thus, for $\| b \|$ small enough, the exponential boundedness of $\mu$ implies that the respective series are absolutely convergent.  We may therefore conclude that the $t$ coefficients of the series expansions agree, provided that $b \in B_{\delta}(0)$ for $\delta > 0$ small enough.  Thus,
\[
\frac{d H^{\mu^{\rhd t}}(b)}{d t}  = K^\mu(H^{\mu^{\rhd t}}(b)).
\]
for $b \in B_{\delta}(0)$.

To extend to arbitrary bounded  sets in $\mathcal{B}^{-}$, consider the net of difference quotients
$$ D^{\mu}_{h}(b,t) =   \frac{H^{\mu^{\rhd t + h}}(b)  - H^{\mu^{\rhd t}}(b) }{h} $$
for $t > 0$.
We have just shown that $$   \lim_{h \rightarrow0}D^{\mu}_{h}(b,t) \rightarrow  K^\mu(H^{\mu^{\rhd t}}(b))$$
uniformly on $B_{\delta}(0)$.  By Theorem $2.10$ in \cite{Belinschi-Popa-Vinnikov-ID}, this implies that the same is true on all bounded sets in $\mathcal{B}^{-}$.
Thus, at the level of functions,
\[
\frac{d H^{\mu^{\rhd t}}(b)}{d t}  = K^\mu(H^{\mu^{\rhd t}}(b)),
\]
proving our result.
\end{proof}

\begin{Cor}
\[
H^{(1_n \otimes \mu)^{\rhd (s + t)}}(b) = H^{(1_n \otimes \mu)^{\rhd s}} \left( H^{(1_n \otimes \mu)^{\rhd t}}(b) \right).
\]
In particular,
\[
F^{\mu^{\rhd (s + t)}}(b) = F^{\mu^{\rhd s}} \left( F^{\mu^{\rhd t}}(b) \right),
\]
so the combinatorial definition of monotone convolution powers coincides with the complex analytic one in Definition~\ref{Defn:Monotone-convolution}.
\end{Cor}

\begin{proof}
By Proposition~\ref{Prop:Evolution}, $H^{\mu^{\rhd s}} \left( H^{\mu^{\rhd t}}(b) \right)$, as a function of $s$, satisfies
\[
\frac{d}{d s} H^{\mu^{\rhd s}} \left( H^{\mu^{\rhd t}}(b) \right) = K^\mu \left(H^{\mu^{\rhd s}} \left( H^{\mu^{\rhd t}}(b) \right) \right), \qquad \left. H^{\mu^{\rhd s}} \left( H^{\mu^{\rhd t}}(b) \right) \right|_{s=0} = H^{\mu^{\rhd t}}(b).
\]
Since, by the same proposition, $H^{\mu^{\rhd (s + t)}}(b)$ also satisfies this differential equation with this initial condition, they coincide for all positive $s$.

For the second statement, we observe that
\[
G^{\mu^{\rhd s}} \left( F^{\mu^{\rhd t}}(b) \right)
= G^{\mu^{\rhd s}} \left( \left(G^{\mu^{\rhd t}}(b) \right)^{-1} \right)
= H^{\mu^{\rhd s}} \left( H^{\mu^{\rhd t}}(b^{-1}) \right)
= H^{\mu^{\rhd (s + t)}}(b^{-1})
= G^{\mu^{\rhd (s + t)}}(b).
\]
\end{proof}

\begin{Prop}
If $\mu, \nu \in \Sigma_0$ and $\mu \rhd \mu = \nu \rhd \nu$, then $\mu = \nu$. In particular, if the square root with respect to the monotone convolution exists, it is unique.
\end{Prop}

\begin{proof}
Under the given assumption,
\[
K^\mu = \frac{1}{2} K^{\mu \rhd \mu} = K^\nu,
\]
and therefore $\mu = \nu$.
\end{proof}

\begin{Remark}
Let $\gamma \in \mc{B}$ be self-adjoint, and $\sigma : \mc{B} \langle X \rangle \rightarrow \mc{B}$ be a completely positive\ but \emph{not} necessarily a $\mc{B}$-bimodule map. Define $\nu_{\uplus}^{\gamma, \sigma}$ via its Boolean cumulant functional
\[
B^{\nu_{\uplus}^{\gamma, \sigma}}[b_0 X b_1] = b_0 \gamma b_1, \quad B^{\nu_{\uplus}^{\gamma, \sigma}}[b_0 X b_1 X \ldots b_{n-1} Xb_n] = b_0 \sigma[b_1 X \ldots b_{n-1}] b_n.
\]
It is known \cite{Belinschi-Popa-Vinnikov-ID,Ans-Bel-Fev-Nica} that $\nu_{\uplus}^{\gamma, \sigma}$ is a completely positive\ $\mc{B}$-bimodule map. Similarly, define $\nu_{\rhd}^{\gamma, \sigma}$ via its monotone cumulant functional
\[
K^{\nu_{\rhd}^{\gamma, \sigma}}[b_0 X b_1] = b_0 \gamma b_1, \quad K^{\nu_{\rhd}^{\gamma, \sigma}}[b_0 X b_1 X \ldots b_{n-1} Xb_n] = b_0 \sigma[b_1 X \ldots b_{n-1}] b_n.
\]
We could also define $\nu_{\boxplus}^{\gamma, \sigma}$ via its free cumulant functional
\[
R^{\nu_{\boxplus}^{\gamma, \sigma}}[b_0 X b_1] = b_0 \gamma b_1, \quad R^{\nu_{\boxplus}^{\gamma, \sigma}}[b_0 X b_1 X \ldots b_{n-1} Xb_n] = b_0 \sigma[b_1 X \ldots b_{n-1}] b_n.
\]
\end{Remark}

\begin{Lemma}
Let $k_i \rightarrow \infty$. For linear $\mc{B}$-bimodule maps $\mu_i : \mc{B} \langle X \rangle \rightarrow \mc{B}$ and $\rho : \mc{B}_0 \langle X \rangle \rightarrow \mc{B}$, the following are equivalent.
\begin{enumerate}
\item
$k_i \mu_i |_{\mc{B}_0} \rightarrow \rho$.
\item
$k_i R^{\mu_i} \rightarrow \rho$.
\item
$k_i B^{\mu_i} \rightarrow \rho$.
\item
$k_i K^{\mu_i} \rightarrow \rho$.
\end{enumerate}
\end{Lemma}

\begin{proof}
We will prove the equivalence between (a) and (d); the rest are similar, and were proved in \cite{Belinschi-Popa-Vinnikov-ID}. Indeed, on $\mc{B}_0 \langle X \rangle$,
\[
\begin{split}
k_i \mu_i[b_0 X  b_1 X \ldots b_{n-1} X b_n]
& = k_i K^{\mu_i}[b_0 X  b_1 X \ldots b_{n-1} X b_n] \\
&\quad + \sum_{\substack{\pi \in \NC(n) \\ \abs{\pi} \geq 2}} \frac{1}{k_i^{\abs{\pi}-1}} \frac{o(\pi)}{\abs{\pi}!} \left(k_i K^{\mu_i}\right)_\pi[b_0 X  b_1 X \ldots b_{n-1} X b_n].
\end{split}
\]
It follows immediately that (d) implies (a). The converse implication follows by induction on $n$.
\end{proof}

\begin{Cor}
For linear $\mc{B}$-bimodule maps $\mu_i : \mc{B} \langle X \rangle \rightarrow \mc{B}$, the following are equivalent.
\begin{enumerate}
\item
\[
k_i \mu_i[X] \rightarrow \gamma, \quad k_i \mu_i[X b_1 X \ldots b_{n-1} X] \rightarrow \sigma[b_1 X \ldots b_{n-1}].
\]
\item
\[
\mu_i^{\boxplus k_i} \rightarrow \nu_{\boxplus}^{\gamma, \sigma}.
\]
\item
\[
\mu_i^{\uplus k_i} \rightarrow \nu_{\uplus}^{\gamma, \sigma}.
\]
\item
\[
\mu_i^{\rhd k_i} \rightarrow \nu_{\rhd}^{\gamma, \sigma}.
\]
\end{enumerate}
\end{Cor}

\begin{proof}
We will prove the equivalence between (a) and (d); the rest are similar, see Lecture~13 in \cite{Nica-Speicher-book}. Indeed, by Lemma~\ref{Lemma:Continuity}, the statement in part (d) is equivalent to
\[
k_i K^{\mu_i} \rightarrow K^{\nu_{\rhd}^{\gamma, \sigma}},
\]
which by definition of $\nu_{\rhd}^{\gamma, \sigma}$ means
\[
k_i K^{\mu_i}[X] \rightarrow \gamma, \quad k_i K^{\mu_i}[X b_1 X \ldots b_{n-1} X] \rightarrow \sigma[b_1 X \ldots b_{n-1}]
\]
This is equivalent to (a) by the preceding lemma.
\end{proof}

\begin{Cor}
$\nu_{\rhd}^{\gamma, \sigma}$ is a completely positive\ map.
\end{Cor}

\begin{proof}
We can choose completely positive\ $\mu_i$ such that $\mu_i^{\uplus i} \rightarrow \nu_{\uplus}^{\gamma, \sigma}$, for example by taking $\mu_i = \nu_{\uplus}^{\frac{1}{i} \gamma, \frac{1}{i}\sigma}$. Then $\nu_{\rhd}^{\gamma, \sigma}$ is the limit of completely positive\ maps $\mu_i^{\rhd i}$, and as such is completely positive (monotone convolution of two completely positive\ maps is known to be positive, see Proposition~6.2 of \cite{Popa-Monotonic-Cstar} and also \cite{Popa-Conditionally-monotone}).
\end{proof}

\begin{Prop}
Monotone convolution semigroups of completely positive\ $\mc{B}$-bimodule maps are in a one-to-one correspondence with pairs $(\gamma, \sigma)$ as above.
\end{Prop}

\begin{proof}
$\set{ \nu_{\rhd}^{ t \gamma, t \sigma} : t \geq 0}$ form a one-parameter monotone convolution semigroup of completely positive\ $\mc{B}$-bimodule maps. Conversely, if $\set{\mu_t}$ is such a semigroup, define
\[
\gamma = \left.\frac{d}{dt} \right|_{t=0} \mu_t[X] = K^{\mu_1}[X] \in \mc{B}^{sa},
\]
\[
\quad \sigma[b_1 X \ldots b_{n-1}] = \left.\frac{d}{dt} \right|_{t=0} \mu_t[X b_1 X \ldots b_{n-1} X] = K^{\mu_1}[X b_1 X \ldots b_{n-1} X].
\]
Since for $P_i \in \mc{B} \langle X \rangle$ and $c_i \in \mc{B}$,
\[
\sum_{i, j=1}^N c_i^\ast \sigma[P_i^\ast P_j] c_j
= \left.\frac{d}{dt} \right|_{t=0} \mu_t \left[ \sum_{i,j=1}^N c_i^\ast X P_i^\ast P_j X c_j \right]
= \lim_{t \downarrow 0} \frac{1}{t} \mu_t \left[ \sum_{i,j=1}^N c_i^\ast X P_i^\ast P_j X c_j \right] \geq 0,
\]
$\sigma$ is completely positive
\end{proof}

\begin{Remark}
A short calculation shows that
\[
\Phi(b) = \gamma + G_\sigma(b).
\]
This, combined with Theorem~\ref{nevanlinna}, gives an alternative proof of the result in Theorem~\ref{generation} that generators of semigroups arising from $\rhd$-infinitely divisible distributions coincide with the set $\Lambda$. One can also use a standard combinatorial argument to show that $\rhd$-infinitely divisible distributions belong to such one-parameter semigroups. At this point, we do not know how to obtain the more general results in Theorem~\ref{generation} by combinatorial methods.
\end{Remark}




\appendix

\section{Characterization of general Cauchy transforms}\label{appendix}

In this appendix, we extend the main result in \cite{Williams-Analytic}, namely the classification of the Cauchy transforms associated to distributions $\mu \in \Sigma_{0}$, to the Cauchy transforms associated to more general CP maps.

\begin{theorem}\label{generalization}
The following are equivalent:
\begin{enumerate}[(I)]
\item\label{tfcae1}  The analytic non-commutative function $G = (G^{(n)})_{n \geq 1 } :  H^{+}(\mathcal{B})  \rightarrow  H^{-}(\mathcal{B}) $ has the property that $H =  (H^{(n)})_{n \geq 1 } $ defined through the equalities $H^{(n)}(b) := G^{(n)}(b^{-1})$ for all $n \in \mathbb{N}$ and $b \in M_{n}(\mathcal{B})$ has uniformly analytic extension to a neighborhood of $0$
satisfying $H^{(n)}(0) = 0$.
\item\label{tfcae2} There exists a $\mathbb{C}$-linear map $\sigma : \mathcal{B} \langle X \rangle \rightarrow \mathcal{B}$
satisfying \eqref{bound} and \eqref{CP} such that $G^{(n)}(b) = \sigma((b-X)^{-1})$.
\end{enumerate}
\end{theorem}
\begin{proof}
We begin with \eqref{tfcae2} $\Rightarrow$ \eqref{tfcae1}.  Let $\sigma$ satisfy \eqref{bound} and \eqref{CP}.  By \cite{Popa-Vinnikov-NC-functions}, Lemma $5.8$, we may conclude that there exists a $\boxplus$-infinitely divisible distribution $\mu \in \Sigma_{0}$ such that $\rho_{\mu}(XP(X)X) = \sigma(P(X))$ for all $P(X) \in \mathcal{B}\langle X \rangle$ (here, $\rho_{\mu}$ denotes the free cumulant function associated to $\mu$).
Thus, the Voiculescu transform of $\mu$ satisfies the following equality:
\begin{equation}\label{voiculescu} \varphi_{\mu}^{(n)} (b) = -\sigma ((b-X)^{-1}) \end{equation}
for all $n \in \mathbb{N}$ and where the inverse in the equality is considered as a geometric series, so that the right hand side is convergent for $\| b^{-1} \|$ small enough dependent on \eqref{bound}.  Since $\mu$ is $\boxplus$-infinitely divisible, by Proposition $5.1$ in \cite{Williams-Analytic}, we have that the left hand side of \eqref{voiculescu} extends to $$H^{+}(\mathcal{B}) \cup H^{-}(\mathcal{B}) \bigcup_{n=1}^{\infty} \{ b \in M_{n}(\mathcal{B}) :\| b^{-1} \| < C\} $$ where $C$ is a fixed constant, independent of $n$.

Now, by Proposition $1.2$ in \cite{Popa-Vinnikov-NC-functions}, the fact that $\mu \in \Sigma_{0}$ implies that $\mu$ is realized as the distribution arising from a non-commutative probability space $(\mathcal{A}, E, \mathcal{B})$.  That is, $$\mu(P(X)) = E(P(a))$$
for a fixed self-adjoint element $a \in \mathcal{B}$ and all $P(X) \in \mathcal{B}\langle X \rangle$.
Thus, $\sigma((b-X)^{-1}) = \rho_{\mu}(a(b-a)^{-1}a)$ and, since $b-a \in M_{n}^{+}(\mathcal{B})$ and $ \rho_{\mu}$ is a CP map on $\mathcal{B}\langle X \rangle_{0}$ we may conclude that the $\sigma((b-X)^{-1}) \in M_{n}^{-}(\mathcal{B})$ for all $b \in M_{n}^{+}(\mathcal{B})$.

 Further note that $$ H(b) =  \sigma((b^{-1}-X)^{-1}) = \sum_{k=0}^{\infty} \sigma((bX)^{k}b)$$
is convergent in a neighborhood of zero since $\sigma$ satisfies \eqref{bound}.  It is also immediate that $H(0) = 0$.
This completes one direction of our proof.

We now prove \eqref{tfcae1} $\Rightarrow $ \eqref{tfcae2}.
We will follow the proof of Theorem 4.1 in \cite{Williams-Analytic} and refer to this paper for the appropriate terminology.

We recover our operator $\sigma$ through the differential structure of $H$.  Indeed, we define the map $\sigma$ by letting
$$(\sigma \otimes 1_{n}) (b_{1} (X \otimes 1_{n}) b_{2} \cdots (X \otimes 1_{n}) b_{\ell+1}) := \Delta_{\mathcal{R}}^{\ell + 1}H^{(n)}(\underbrace{0,\ldots,0}_{\ell + 2 \  - \ \text{times}})( b_{1} , b_{2} , \ldots , b_{\ell+1})$$
for elements $b_{1} , b_{2} , \cdots , b_{\ell + 1} \in M_{n}(\mathcal{B})$.  It is a consequence of Proposition $3.1$ in \cite{Williams-Analytic} and \cite{KVV}, Theorem 3.10 that this is a well defined operator.  Moreover, the equality
$$ \Delta_{\mathcal{R}}^{\ell + 1}H^{(n)}(\underbrace{0,\ldots,0}_{\ell + 2 \  - \ \text{times}})( b , b , \ldots , b) = \frac{1}{(\ell + 1)!}
\frac{d^{\ell+1}}{dt^{\ell+1}} H^{(n)}(0 + tb)|_{t=0}$$
 and the fact that the function is analytic in a neighborhood of $0$ implies that
\begin{equation}\label{Hseries}
H^{(n)} (b) = \sum_{k=0}^{\infty} (\sigma \otimes 1_{n}) ((bX)^{k}b)
\end{equation}
once we show that $\sigma$ satisfies \eqref{bound}.
Continuation will allow us to conclude that \begin{equation}\label{Gseries}
G^{(n)}(b) = \sum_{k=0}^{\infty} (\sigma \otimes 1_{n}) ((b^{-1}X)^{k}b^{-1}) = (\sigma \otimes 1_{n}) ((b - X)^{-1}).
\end{equation}

Thus, our theorem will follow when we can show that
$\sigma$ satisfies properties \eqref{bound} and
 \eqref{CP}.

To prove \eqref{bound}, we note that this is equivalent to showing that
 $$\| \sigma(b_{1}Xb_{2} \cdots Xb_{\ell + 1}) \|  \leq CM^{\ell +1}$$
for a fixed $C > 0$, provided that $\| b_{1} \| = \cdots = \| b_{\ell + 1} \| = 1$.   This will follow from uniform analyticity and matches the proof of the same fact in \cite{Williams-Analytic}.  Indeed,
 consider the element of $ M_{\ell + 2}(\mathcal{B})$
$$ B = \left( \begin{array} {cccccc}
0 & b_{1} & 0 & 0 & \cdots & 0 \\
0 & 0 & b_{2} & 0 & \cdots & 0 \\
0 & 0 & 0 & b_{3} & \cdots & 0 \\
\ & \vdots & \ & \ & \vdots & \ \\
0 & 0 & 0 & 0 & \cdots & b_{\ell+1} \\
0 & 0 & 0 & 0 & \cdots &0
\end{array} \right) .$$
Note that $H^{(\ell+1)}$ has a bound of $C$ on a ball of radius $r$ about $0$, independent of $\ell$ since we are assuming that $H$ is uniformly analytic.
Thus, \begin{align*}
      \| \sigma(b_{1}Xb_{2} \cdots X b_{\ell + 1}) \| & = \frac{ \| \delta^{\ell + 1}H^{(\ell + 2)}(0;B) \| }{(\ell + 1)!}\\
& = \| \Delta_{\mathcal{R}}^{\ell + 1}H^{(\ell + 2)} (0 , \ldots , 0)( B , \ldots , B) \| \\
& =  \| r^{-(\ell + 1)} \Delta_{\mathcal{R}}^{\ell + 1} H^{(\ell + 2)} (0 , \ldots , 0 )( r B , \ldots , r B) \| \\
& =   \left( \frac{1}{r} \right)^{\ell + 1} \frac{ \| \delta^{\ell + 1}H^{(\ell + 2)}(0;rB) \| }{(\ell + 1)!} \\
& \leq C \left( \frac{1}{r} \right)^{\ell + 1}
      \end{align*}
 where the last inequality follows from the Cauchy estimates in Theorem~\ref{Cauchy}.

We must prove the technical fact that fact that \begin{equation}\label{gapbridge}
\sigma|_{M_{n}(\mathcal{B})}\geq 0 \end{equation}
Assume that $\sigma(P) < 0$ for some $P \in M_{n}^{+}(\mathcal{B})$ where we can assume that $P > \delta 1$ for some $\delta > 0$.
Note that $G^{(n)}(zP^{-1}) \in M_{n}^{-}(\mathcal{B})$ for all $z\in \mathbb{C}^{+}$  by assumption so that $\lambda G^{(n)}(i\lambda P^{-1}) \in M_{n}^{-}(\mathcal{B})$ for all $\lambda \in \mathbb{R}^{+}$.
Utilizing the series expansion in \eqref{Gseries} as well as the exponential bound that we have just proven, we conclude that the
$$ \lim_{\lambda \uparrow \infty} \lambda G^{(n)}(i\lambda P^{-1}) = \frac{\sigma(P)}{i} = -i\sigma(P) \notin M_{n}^{-}(\mathcal{B}).$$
This contradiction implies \eqref{gapbridge}.

It remains to show \eqref{CP}.  Once again, this will closely follow the proof of the analogous fact in Theorem 4.1 in \cite{Williams-Analytic}.
Indeed, we will first show that \begin{equation}\label{mon_claim} (\sigma\otimes 1_{n}) (P(X \otimes 1_{n} + b_{0})^{\ast} P(X \otimes 1_{n} + b_{0})) \geq 0 \end{equation}
for any monomial $P(X) = b_{1} (X \otimes 1_{n}) b_{2} \cdots X \otimes 1_{n} b_{\ell+ 1} \in M_{n}(\mathcal{B})\langle X \rangle$ and $b_{0} \in  M_{n}(\mathcal{B})$.   We also assume that $|b_{\ell + 1}| > \epsilon 1_{n}$ and the general case follows by letting $\epsilon \downarrow 0$.

Towards this end,
we consider  elements $ C , E_{0}, E_{1} \in M_{n(\ell + 1)}(\mathcal{B})$ defined as follows:
$$ C =  \left( \begin{array}  {ccccccc}
    0 & c_{1} & 0 & 0 & 0 &  \cdots & 0 \\
 c_{1}^{\ast} & 0 & c_{2} & 0 & 0 &\cdots & 0 \\
0 & c_{2}^{\ast} & 0 & c_{3} & 0 & \cdots & 0 \\
\vdots & \ & \ & \vdots & \ & \ & \vdots \\
0 & 0 & \cdots & 0 & c_{\ell - 1}^{\ast} & 0 & c_{\ell} \\
0 & 0 & \cdots & 0 & 0 & c_{\ell}^{\ast} & |c_{\ell + 1}|^{2}
   \end{array} \right)  ; \ \ E_{0} =\underbrace{1_{n} \oplus 1_{n} \oplus \cdots \oplus 1_{n}}_{\ell \ times} \oplus 0_{n}$$
and  $E_{1} = 1_{n(\ell + 1)} - E_{0}$ where $c_{i} = \delta b_{i}$ for $i = 1 , \ldots , \ell$ and $c_{\ell + 1} = b_{\ell + 1}/\delta^{\ell}$ for $\delta > 0$ to be specified.  Note that $b_{1}Xb_{2} \cdots Xb_{\ell + 1} =c_{1}Xc_{2} \cdots Xc_{\ell + 1} $.
We define a function
$$\hat{g}^{n(\ell+1)}(b) := G^{n(\ell+1)}(b - b_{0}) : M_{n(\ell + 1)}^{+}(\mathcal{B}) \rightarrow  M_{n(\ell + 1)}^{-}(\mathcal{B})$$
The following properties are rather trivial and their proof matches those of Theorem 4.1 in \cite{Williams-Analytic}.
\begin{enumerate}[(a)]
\item\label{fact1} $C + \epsilon E_{0} > \gamma 1_{n}$  for some $\gamma > 0$ provided that $\delta > 0$ is small enough.
\item\label{fact2} The $n \times n$ minor in the top left corner of $$[(C + \epsilon E_{0})(X \otimes 1_{n(\ell + 1)} + b_{0} \otimes 1_{\ell + 1})]^{2(\ell-1) } (C + \epsilon E_{0})$$ is equal to $P(X +  b_{0})P^{\ast}(X +  b_{0}) + O(\epsilon)$.
\item\label{fact3} $\hat{g}^{(n(\ell + 1))}(b) = \sum_{p=0}^{\infty} \sigma( [ b^{-1}(X \otimes 1_{n(\ell + 1)} + b_{0} \otimes 1_{\ell + 1}) ]^{p} b^{-1}) $ for $b^{-1}$ in a neighborhood of $0$.
\item\label{fact4} We have that $ z\hat{g}^{(n(\ell + 1))}(z b) \rightarrow \sigma(b^{-1})$ in norm as $|z| \uparrow \infty$ for $b > \gamma1_{n}$.
\item\label{fact5} $\hat{h}^{(n(\ell + 1))}(b):=\hat{g}^{(n(\ell + 1))}(b^{-1})$ has analytic extension to a neighborhood of zero.
\end{enumerate}
The only one of these properties that differs from the proof of Theorem 4.1 in \cite{Williams-Analytic} is \eqref{fact4}.  It follows immediately from the series expansion in \eqref{Hseries}.

We now have the pieces in place to prove \eqref{mon_claim}.
Note that \eqref{fact1} implies that  $C + \epsilon E_{0}$ is  invertible so that  the map $$z \mapsto \hat{g}^{(n(\ell + 1))}(z(C + \epsilon E_{0})^{-1})$$  sends $\mathbb{C}^{+}$ into $M_{n}(\mathcal{B})^{-}$.
Let $B_{i,j} \in M_{n}(\mathcal{B})$ for $i,j = 1 , \ldots , \ell + 1$ and consider the element $B = (B_{i,j})_{i,j = 1}^{\ell + 1} \in M_{n(\ell + 1)}(\mathcal{B})$.
Given a state $f \in M_{n}(\mathcal{B})^{\ast}$ we define a new state $$f_{1,1}(B) := f(B_{1,1}) :   M_{n(\ell + 1)}(B) \rightarrow \mathbb{C}.$$
We may define a map $$G_{f,C,\epsilon}(z) = f_{1,1} \circ g^{(n(\ell + 1))}(z(C + \epsilon E_{0})^{-1}) : \mathbb{C}^{+} \rightarrow \mathbb{C}^{-}.$$
Properties \eqref{fact3} and \eqref{fact4} imply the following for $z \in \mathbb{C}^{+}$:
\begin{align*}
\lim_{|z| \uparrow \infty} z G_{f,C,\epsilon}(z) = & \lim_{|z| \uparrow \infty} f_{1,1} \left[  z g^{(n(\ell+1))}(z(C + \epsilon E_{0})^{-1} )  \right] \\
& = f_{1,1}( \sigma( C + \epsilon E_{0}) ) \geq 0
\end{align*}
where the last inequality will follow from the fact that  $f_{1,1}$ is a state, property \eqref{fact1}
and \eqref{gapbridge}.

Now, observe that the coefficient of $z^{-2\ell + 1}$ in the function $G_{f,C,\epsilon}$ is equal to $\rho(t^{2(\ell-1)}) > 0$.
Furthermore, since \begin{align*}
G_{f,C,\epsilon}(z) & = G_{\rho}(z) = \sum_{\ell = 0}^{\infty} \frac{\rho(t^{\ell})}{z^{\ell + 1}} \\
& = \sum_{\ell = 0}^{\infty} \frac{f_{1,1}(\sigma([(C + \epsilon E_{0})(X \otimes 1_{n(\ell + 1)} + b_{0})]^{\ell} (C + \epsilon E_{0}) ))}{z^{\ell + 1}}
\end{align*}
we may conclude that $$f_{1,1} \circ \sigma([(C+ \epsilon E_{0})(X \otimes 1_{n(\ell + 1)} + b_{0})]^{2(\ell-1)} (C + \epsilon E_{0}) )= \rho(t^{2(\ell - 1)}) \geq 0.$$
Recalling \eqref{fact2}, it follows that $f \circ \sigma ( [P(X + b_{0})P^{\ast}(X +  b_{0}) + O(\epsilon)]) \geq 0$.  Letting $\epsilon \downarrow 0$ and noting that $f$ was an arbitrary state, we have proven that
$$ (\sigma \otimes 1_{n}) ( P(X + b_{0})P^{\ast}(X +  b_{0}) ) \geq 0$$
for any monomial $P(X) \in M_{n}(\mathcal{B})\langle X \rangle$.

The extension from the case of monomials to general elements in $\mathcal{B}\langle X \rangle$ follows the proof in \cite{Williams-Analytic} exactly so we will refrain from repeating it.  This implies \eqref{CP} and, therefore, our theorem.

\end{proof}


\def\cprime{$'$} \def\cprime{$'$}
\providecommand{\bysame}{\leavevmode\hbox to3em{\hrulefill}\thinspace}
\providecommand{\MR}{\relax\ifhmode\unskip\space\fi MR }
\providecommand{\MRhref}[2]{%
  \href{http://www.ams.org/mathscinet-getitem?mr=#1}{#2}
}
\providecommand{\href}[2]{#2}

\end{document}